\documentclass{amsart}

\usepackage{amsmath,amsthm,amstext,amssymb,bbm}
\usepackage{hyperref}
\usepackage{accents}
\usepackage[utf8]{inputenc}
\usepackage{enumerate,enumitem}
\usepackage{amscd,  amsfonts,  amsbsy,  mathrsfs, upgreek, mathtools, stmaryrd, bbm}

\usepackage{geometry}
\usepackage{todonotes}

\theoremstyle{plain}
\newtheorem{theorem}{Theorem}[section]
\newtheorem{lemma}[theorem]{Lemma}

\newtheorem*{observation*}{Observation}
\newtheorem{corollary}[theorem]{Corollary}
\newtheorem{proposition}[theorem]{Proposition}

\newtheorem*{claim*}{Claim}
\newtheorem*{subclaim*}{Subclaim}
\theoremstyle{definition}
\newtheorem{definition}[theorem]{Definition}

\newcommand{\betrag}[1]{\vert{#1}\vert}

\newcommand{\crit}[1]{{{\rm{crit}}\left({#1}\right)}}
\newcommand{\cof}[1]{{{\rm{cof}}(#1)}}
\newcommand{\otp}[1]{{{\rm{otp}}\left(#1\right)}}
\newcommand{\ran}[1]{{{\rm{ran}}(#1)}}

\newcommand{\supp}[1]{{{\rm{supp}}(#1)}}

\newcommand{\tc}[1]{{\rm{tc}}({#1})}

\newcommand{\POT}[1]{{\mathcal{P}}({#1})}

\newcommand{\map}[3]{{#1}:{#2}\longrightarrow{#3}}
\newcommand{\Map}[5]{{#1}:{#2}\longrightarrow{#3};~{#4}\longmapsto{#5}}

\newcommand{\Set}[2]{\{{#1}~\vert~{#2}\}}
\newcommand{\seq}[2]{\langle{#1}~\vert~{#2}\rangle}

\newcommand{\goedel}[2]{{\prec}{#1},{#2}{\succ}}

\newcommand{\anf}[1]{{\text{``}\hspace{0.3ex}{#1}\hspace{0.3ex}\text{''}}}

\newcommand{\HH}[1]{{\rm{H}}(#1)}

\newcommand{\Add}[2]{{\rm{Add}}({#1},{#2})}

\newcommand{\SR}[2]{{\rm{SR}}_{{#1}}({#2})}
\newcommand{\SRm}[2]{\rm{SR}^-_{{#1}}({#2})}

\newcommand{\id}{{\rm{id}}}
\newcommand{\Lim}{{\rm{Lim}}}
\newcommand{\On}{{\rm{Ord}}}
\newcommand{\LL}{{\rm{L}}}

\newcommand{\ZFC}{{\rm{ZFC}}}

\newcommand{\CCCC}{{\mathsf{C}}}

\newcommand{\VV}{{\rm{V}}}

\newcommand{\calC}{\mathcal{C}}

\newcommand{\calK}{\mathcal{K}}
\newcommand{\calL}{\mathcal{L}}

\newcommand{\calV}{\mathcal{V}}
\newcommand{\calW}{\mathcal{W}}

\makeatletter
\@namedef{subjclassname@2020}{%
  \textup{2020} Mathematics Subject Classification}
\makeatother

\title[Structural reflection, shrewd cardinals and the continuum]{Structural reflection, shrewd cardinals and the size of the continuum}

\author{Philipp L\"ucke}
\address{Institut de Matem\`{a}tica, Universitat de Barcelona.  Gran via de les Corts Catalanes 585, 08007 Barcelona, Spain.}
\email{philipp.luecke@ub.edu}


\thanks{The author would like to thank Joan Bagaria for several helpful discussions on the topic of this work and various comments on earlier versions of this paper. 
 In addition, the author is thankful to the anonymous referee for the careful reading of the manuscript and several helpful comments.  
 This project has received funding from the European Union’s Horizon 2020 research and innovation programme under the Marie Sk{\l}odowska-Curie grant agreement No 842082 (Project \emph{SAIFIA: Strong Axioms of Infinity -- Frameworks, Interactions and Applications}).}

\begin{document}

\begin{abstract}
 Motivated by results of Bagaria, Magidor and V\"a\"an\"anen, we study characterizations of large cardinal properties through reflection principles for classes of structures. 
  More specifically, we aim to characterize notions from the lower end of the large cardinal hierarchy through the principle $\mathrm{SR}^-$ introduced by Bagaria and V\"a\"an\"anen. 
  Our results isolate a narrow interval in the large cardinal hierarchy that is bounded from below by total indescribability and from above by subtleness, and contains all large cardinals that can be characterized through the validity of the principle $\mathrm{SR}^-$ for all classes of structures defined by formulas in a fixed level of the L\'{e}vy hierarchy. 
 Moreover, it turns out that no property that can be characterized through this  principle can provably imply strong inaccessibility. 
 The proofs of these results  rely heavily on the notion of  \emph{shrewd cardinals}, introduced by Rathjen in a proof-theoretic context, and  embedding characterizations of these cardinals that resembles Magidor's classical characterization of supercompactness. 
 In addition, we show that several important weak large cardinal properties, like weak inaccessibility, weak Mahloness or  weak $\Pi^1_n$-indescribability, can be canonically characterized through localized versions of the principle $\mathrm{SR}^-$. 
 Finally, the techniques developed in the proofs of these characterizations also allow us to show that Hamkin's \emph{weakly compact embedding property} is equivalent to L\'{e}vy's notion of weak $\Pi^1_1$-indescribability. 
\end{abstract}

\maketitle




\section{Introduction}\label{section:Intro}

The work presented in this paper is motivated by results of Bagaria, Magidor, V\"a\"an\"anen and others that establish deep connections between  extensions of the \emph{Downward L\"owenheim--Skolem Theorem}, large cardinal axioms and set-theoretic reflection principles. 
  Our results will focus on the characterization of large cardinal notions\footnote{Throughout this paper, we will use the term \emph{large cardinal notion} to refer to properties of cardinals that imply \emph{weak} inaccessibility.}  through reflection properties for classes of structures. 
 To motivate our work, we start by discussing results that provide such  characterizations for supercompact cardinals.

First, recall that classical work of Magidor in \cite{MR0295904} yields a characterization of supercompactness through model-theoretic reflection by showing that second-order logic has a \emph{L\"owenheim--Skolem--Tarski cardinal} (see {\cite[Definition 6.1]{zbMATH06618251}) if and only if there exists a supercompact cardinal. 
Moreover, Magidor's results show that if these equivalent statements hold true, then the least L\"owenheim--Skolem--Tarski cardinal for second-order logic is equal to the least supercompact cardinal.  
 Next, in order to connect supercompactness and reflection principles for second-order logic to set-theoretic reflection properties, we make use of the following principle of \emph{structural reflection}, formulated by Bagaria:

 \begin{definition}[Bagaria, \cite{zbMATH06029795}]
 Given an infinite cardinal $\kappa$ and a class $\calC$ of structures\footnote{In the following, the term \emph{structure}  refers to structures for countable first-order languages. It should be noted that all results cited and proven below remain true if we restrict ourselves to finite languages.} of the same type, 
  we let $\SR{\calC}{\kappa}$ denote the statement that 
  for every structure $A$ in $\calC$, there exists an elementary embedding of a structure in $\calC$ of cardinality less than $\kappa$ into $A$. 
\end{definition}

 In order to precisely formulate the relevant results from \cite{zbMATH06029795}, \cite{zbMATH06424459} and \cite{zbMATH06618251}, we first need to  discuss certain subclasses of $\Sigma_2$-formulas defined through standard  refinements of  the \emph{L{\'e}vy hierarchy} of formulas. 
 Given a first-order language $\calL$ that extends the language $\calL_\in$ of set theory, an $\calL$-formula is a \emph{$\Sigma_0$-formula} if it is contained in the smallest class of $\calL$-formulas that contains all atomic $\calL$-formulas and is closed under negations, conjunctions and bounded existential quantification. 
 Moreover, given $n<\omega$,  an $\calL$-formula is a \emph{$\Pi_n$-formula} if it is the negation of a $\Sigma_n$-formula, and it is a \emph{$\Sigma_{n+1}$-formula} if it is of the form $\exists x_0,\ldots,x_{m-1} ~\varphi$ for some $\Pi_n$-formula $\varphi$. 
 Finally, given a class $R$, $n<\omega$ and a set $z$, a class $P$ is \emph{$\Sigma_n(R)$-definable} (respectively, \emph{$\Pi_n(R)$-definable}) \emph{in the parameter $z$} if there is a $\Sigma_n$-formula (respectively, a $\Pi_n$-formula) $\varphi(v_0,v_1)$ in the language $\calL_{\dot{A}}$ that expands $\calL_\in$ by a unary relation symbol $\dot{A}$ such that $P$ consists of all sets $x$ with the property that $\varphi(x,z)$ holds in $\langle\VV,\in,R\rangle$. 
  As usual, we call $\Sigma_n(\emptyset)$-definable classes (i.e. classes defined by a $\Sigma_n$-formula in $\calL_\in$) \emph{$\Sigma_n$-definable} and $\Pi_n(\emptyset)$-definable classes \emph{$\Pi_n$-definable}. 
 The closure properties of the class of all $\calL_\in$-formulas that are $\ZFC$-provably equivalent to $\Sigma_{n+1}$-formulas then ensure that if $R$ is a class that is $\Pi_1$-definable in the parameter $z$, then every class that is $\Sigma_n(R)$-definable in the parameter $z$ is also $\Sigma_{n+1}$-definable in this parameter.

Using V\"a\"an\"anen's notion of \emph{symbiosis} between model- and set-theoretic reflection principles introduced in \cite{MR567682}, Bagaria and V\"a\"an\"anen connected reflection principles for second-order logic with the principle $\mathrm{SR}$ by showing that a cardinal $\kappa$ is a  L\"owenheim--Skolem--Tarski cardinal for second-order logic if and only if $\SR{\calC}{\kappa}$ holds  for every class $\calC$ of structures that is $\Sigma_1(PwSet)$-definable without parameters (see {\cite[Theorem 5.5]{zbMATH06618251}} and {\cite[Lemma 7.1]{zbMATH06618251}}), where $PwSet$ denotes the $\Pi_1$-definable class of all pairs of the form $\langle x,\POT{x}\rangle$. 
 In combination with the results from \cite{MR0295904} discussed above, this equivalence provides a characterization of  the first supercompact cardinal  through the validity of the principle $\SR{\calC}{\kappa}$ for all $\Sigma_1(PwSet)$-definable classes of structures. 
 This connection between supercompactness and structural reflection for $\Sigma_2$-definable classes was studied in depth by Bagaria and his collaborators in \cite{zbMATH06029795} and \cite{zbMATH06424459}.  
  %
 In the following, define $\calV$ to be the class of all $\calL_\in$-structures\footnote{Throughout this paper, we will often identify a class $M$ with the $\calL_\in$-structure $\langle M,\in\rangle$ to simplify our formulations.} of the form $\VV_\alpha$ for some ordinal $\alpha$. 
 It is easy to see that the class $\calV$ is $\Sigma_1(PwSet)$-definable without parameters. 
 Following \cite{zbMATH06029795} and \cite{zbMATH06424459}, Magidor's characterization of supercompactness in \cite{MR0295904} can now 
 be rephrased in the following way:

 \begin{theorem}\label{theorem:CharSuperCom} 
  The following statements are equivalent for every  cardinal $\kappa$: 
     \begin{enumerate}
      \item $\kappa$ is the least supercompact cardinal. 
      
      \item $\kappa$ is the least cardinal with the property that $\SR{\calV}{\kappa}$ holds. 
      
           \item $\kappa$ is the least cardinal with the property that $\SR{\calC}{\kappa}$ holds for every class of structures of the same type that is definable by a $\Sigma_2$-formula with parameters in $\HH{\kappa}$.\footnote{It is easy to see  that  the parameters class $\HH{\kappa}$ is maximal in this setting. Fix $z\notin\HH{\kappa}$ and let $\calL_\emptyset$ denote the trivial first-order language. Then the class $\calC$ of all $\calL$-structures of cardinality $\betrag{\tc{z}}$ is definable by a $\Sigma_1$-formula with parameter $z$, and the principle $\SRm{\calC}{\kappa}$ obviously fails.} 
     \end{enumerate}
 \end{theorem}

 It is  natural to ask if other large cardinal notions can be characterized in similar ways, i.e. given some large cardinal property, is there a class of formulas such that the least cardinal with the given property provably coincides with the least reflection point for all classes of structures defined by formulas from this class. 
 %
  %
  Note that the validity of such characterization can be seen as a strong justification for the naturalness of large cardinal axioms (see \cite{BagariaRefl}). 
 Results contained in \cite{zbMATH06029795} and \cite{zbMATH06424459} already yield such characterizations for $\Sigma_{n+2}$-definable classes of structures and so-called \emph{$C^{(n)}$-extendible cardinals} (see {\cite[Definition 3.2]{zbMATH06029795}}). 
 These results also provide a characterization of \emph{Vop{\v e}nka's Principle} in terms of structural reflection. 
 In addition, results in \cite{BagariaRefl} provide a characterization of the existence of $X^\#$ for a set $X$ of ordinals through structural reflection and results in \cite{BagariaWilsonRefl} yield such characterization for large cardinal notions in the region between strong cardinals and \anf{\emph{$\On$ is Woodin}}. 
 Finally, the results of \cite{MR3598793} give characterizations of remarkable cardinals  and other virtual large cardinals through principles of \emph{generic structural reflection}.

The work presented in this paper is devoted  to the characterization of objects from the lower part of the large cardinal hierarchy through principles of structural reflection. 
The starting point of this work is the following restriction of the principle $\mathrm{SR}$, isolated by Bagaria and V\"a\"an\"anen:

\begin{definition}[Bagaria--V\"a\"an\"anen, \cite{zbMATH06618251}]
 Given an infinite cardinal $\kappa$ and a class $\calC$ of structures of the same type, 
  we let $\SRm{\calC}{\kappa}$ denote the statement that 
  for every structure $A$ in $\calC$ of cardinality $\kappa$, there exists an elementary embedding of a structure in $\calC$ of cardinality less than $\kappa$  into $A$. 
\end{definition}

Our results will show that there exists a narrow interval in the large cardinal hierarchy that contains all large cardinal notions for which there exists a natural number\footnote{Note that {\cite[Theorem 4.2]{zbMATH06029795}} shows that $\SR{\calC}{\kappa}$ holds for every uncountable cardinal $\kappa$ and every class $\calC$ of structures that is definable by a $\Sigma_1$-formula with parameters in $\HH{\kappa}$.} $n>1$ such that the given property can be characterized (as in Theorem \ref{theorem:CharSuperCom}) through the validity of the principle $\mathrm{SR}^-$ for all $\Sigma_n$-definable classes of structures.  
 This interval is bounded from below by total indescribability and from above by subtleness. 
 Moreover, this analysis will show that there is essentially only one large cardinal notion that can be characterized through canonical non-trivial $\Pi_1$-predicates\footnote{More precisely, through $\Pi_1$-predicates $R$ with the property that the class $Cd$ of all cardinals is $\Sigma_1(R)$-definable. The class $Cd$, the class $Rg$ of all regular cardinals, the class $PwSet$ and the class obtained by a universal $\Pi_1$-formula in $\calL_\in$ are all examples of such predicates.} $R$ and the principle $\mathrm{SR}^-$ for $\Sigma_1(R)$-definable classes of structures. 
  This unique large cardinal notion turns out to be closely related to the following property of large cardinals, introduced by Rathjen in a proof-theoretic context:

  \begin{definition}[Rathjen, \cite{zbMATH02168085}]\label{definition:ShrewdCardinals}
   A cardinal $\kappa$ is \emph{shrewd} if for every $\calL_\in$-formula $\Phi(v_0,v_1)$, every ordinal $\alpha$ and every subset $A$ of $\VV_\kappa$ such that $\Phi(A,\kappa)$ holds in $\VV_{\kappa+\alpha}$, there exist ordinals $\bar{\kappa}$ and $\bar{\alpha}$ below $\kappa$ such that $\Phi(A\cap\VV_{\bar{\kappa}},\bar{\kappa})$ holds in $\VV_{\bar{\kappa}+\bar{\alpha}}$. 
\end{definition}

 The defining property of shrewd cardinals directly implies that all of these cardinals are totally indescribable. 
   Moreover,  Rathjen showed that, given a subtle cardinal $\delta$, the set of cardinals $\kappa<\delta$ that are shrewd in $\VV_\delta$ is stationary in $\delta$ (see {\cite[Lemma 2.7]{zbMATH02168085}}). 
The following result now connects shrewdness to the consistency strength of principles the principle $\mathrm{SR}^-$ for $\Sigma_2$-definable classes:

\begin{theorem}\label{theorem:ConsStrengthRefl}
   The following statements are equiconsistent over the theory $\ZFC$:   
   \begin{enumerate}
    \item There exists a shrewd cardinal. 
    
    \item There exists a cardinal $\kappa$ with the property that $\SRm{\calC}{\kappa}$ holds for every class $\calC$ of structures of the same type that is definable by a $\Sigma_1(Cd)$-formula without parameters. 
    
    \item There exists a cardinal $\kappa$ with the property that $\SRm{\calC}{\kappa}$ holds for every class $\calC$ of  structures of the same type that is definable by a $\Sigma_2$-formula with parameters in $\HH{\kappa}$. 
   \end{enumerate}
 \end{theorem}

 This result shows that  for all large cardinal properties whose consistency strength is strictly smaller than the existence of a shrewd cardinal, there is no  characterization of these notions  through canonical non-trivial $\Pi_1$-predicates $R$ and  the principle $\mathrm{SR}^-$ for $\Sigma_1(R)$-definable classes of structures. 
 %
 Moreover, it shows that the connection between the principle $\textrm{SR}^-$ and the \emph{strict L\"owenheim--Skolem--Tarski property} (see {\cite[Definition 8.2]{zbMATH06618251}}) as well as the characterizations of weak inaccessibility, weak Mahloness and weak compactness stated in  {\cite[Theorem 8.3]{zbMATH06618251}} need to be reformulated,\footnote{The problematic part in the adaption of the proof of {\cite[Theorem 5.5]{zbMATH06618251}} to a proof of these statements is the fact that the cardinality of all $\calL_\in$-models witnessing that some structure of cardinality $\kappa$ is contained in a $\Sigma_1(R)$-definable class can be strictly greater than the given cardinal $\kappa$ and therefore it is not possible to apply the strict L\"owenheim--Skolem--Tarski property at $\kappa$ to these $\calL_\in$-models. An example of such a class of structures is the $\Sigma_1(PwSet)$-definable class $\calW$ introduced below.} because, by the above theorem, all of these statements would imply that the consistency of the existence of a shrewd cardinal  is strictly weaker than the consistency strength of the existence of a total indescribable cardinal. 

The proof of Theorem \ref{theorem:ConsStrengthRefl} is based on the following weakening of Definition \ref{definition:ShrewdCardinals}:

\begin{definition}
  An infinite cardinal $\kappa$ is \emph{weakly shrewd} if for every $\calL_\in$-formula $\Phi(v_0,v_1)$, every cardinal $\theta>\kappa$ and every subset $A$ of $\kappa$ with the property that $\Phi(A,\kappa)$ holds in  $\HH{\theta}$, there exist cardinals $\bar{\kappa}<\bar{\theta}$ with the property that $\bar{\kappa}<\kappa$ and $\Phi(A\cap\bar{\kappa},\bar{\kappa})$ holds in $\HH{\bar{\theta}}$. 
\end{definition}

 The notion of weak shrewdness turns out to be closely connected to principles of structural reflection. 
 Our results will allow us to show that if $\kappa$ is a weakly shrewd cardinal, then $\SRm{\calC}{\kappa}$ holds for every class $\calC$ of structures that is $\Sigma_2$-definable with parameters in $\HH{\kappa}$ (see Lemma \ref{lemma:ShrewdReflection}). 
 Moreover, the next theorem shows that this large cardinal property can be canonically characterized through the principle $\mathrm{SR}^-$ for $\Sigma_1(PwSet)$-definable classes of structures.
  In the following, let $\calL_{\dot{c}}$ denote the first-order language that extends the language $\calL_\in$ of set theory by a constant symbol $\dot{c}$ and let $\calW$ denote the class of all $\calL_{\dot{c}}$-structures $\langle X,\in,\kappa\rangle$ with the property that there exists a cardinal $\theta$ such that $\kappa$ is an infinite cardinal smaller than $\theta$ and $X$ is an elementary submodel of $\HH{\theta}$ of cardinality $\kappa$ with $\kappa+1\subseteq X$.  
  It is easy to see that the class $\calW$ is definable by a $\Sigma_1(PwSet)$-formula without parameters. 

 \begin{theorem}\label{theorem:SRandSRCARD}
  The following statements are equivalent for every  cardinal $\kappa$:  
  \begin{enumerate}
   \item $\kappa$ is the least weakly shrewd cardinal. 

   \item $\kappa$ is the least cardinal with the property that $\SRm{\calW}{\kappa}$ holds. 
   
   \item $\kappa$ is the least cardinal with the property that $\SRm{\calC}{\kappa}$ holds for every class $\calC$ of structures of the same type that is definable by a $\Sigma_2$-formula with parameters in $\HH{\kappa}$. 
  \end{enumerate}
  %
 \end{theorem}

 In combination with  Theorem \ref{theorem:ConsStrengthRefl}, this result directly yields the following equiconsistency:

\begin{corollary}
    The following statements are equiconsistent over the theory $\ZFC$:   
   \begin{enumerate}
    \item There exists a shrewd cardinal. 

    \item There exists a weakly shrewd cardinal. \qed 
   \end{enumerate}
\end{corollary}

 In addition, the third statement listed in Theorem \ref{theorem:SRandSRCARD} shows that large cardinal properties of higher consistency strength than  shrewdness cannot be characterized through  the principle $\mathrm{SR}^-$ for $\Sigma_2$-definable classes of structures. 
 Together with our earlier observations, this shows that weak shrewdness is basically the only large cardinal notions that can be characterized with the help of canonical  $\Pi_1$-predicates $R$ and the principle $\mathrm{SR}^-$ for $\Sigma_1(R)$-definable classes of structures.

 The above results directly motivate several follow-up questions. 
 First, it is natural to ask which large cardinal properties stronger than weak shrewdness can be characterized through the principle $\mathrm{SR}^-$ for classes of structures defined by more complex formulas. 
 Second, these results suggest to study the interactions between principles of structural reflection and the behavior of the continuum function.  
 In particular, it is interesting to ask whether any large cardinal property that entails strong inaccessibility can be characterized through the principle $\mathrm{SR}^-$. 
 Finally, it is also natural to ask whether large cardinal notions weaker than shrewdness can be characterized through further restrictions of the principle $\mathrm{SR}^-$.

 The answers to the first two questions turn out to be closely related to the existence of weakly shrewd cardinals that are not shrewd. 
 The following result positions the consistency strength of the existence of weakly shrewd cardinals that are, for various reasons, not shrewd  in the large cardinal hierarchy:

\begin{theorem}\label{theorem:WShrewNonShrewGCH}
 \begin{enumerate} 
  \item If $\kappa$ is a weakly shrewd cardinal that is not shrewd, then there exists an ordinal $\varepsilon>\kappa$ with the property that $\varepsilon$ is inaccessible in $\LL$ and $\kappa$ is a shrewd cardinal in $\LL_\varepsilon$. 
  
  \item The least subtle cardinal is a stationary limit of inaccessible weakly shrewd cardinals that are not shrewd. 
  
  \item The following statements are equiconsistent over $\ZFC$:   
   \begin{enumerate}
    \item There exists an inaccessible weakly shrewd cardinal that is not shrewd. 
  
     \item There exists a weakly shrewd cardinal that is not inaccessible. 
  
    \item There exists a weakly shrewd cardinal smaller than $2^{\aleph_0}$. 
   \end{enumerate}
  \end{enumerate}
\end{theorem}

The techniques developed in the proofs of the above results will also allow us to show that the existence of a weakly shrewd cardinal does not imply the existence of a cardinal $\kappa$ with the property that $\SRm{\calC}{\kappa}$ holds for every class $\calC$ of structures that is definable by a $\Sigma_3$-formula without parameters (see Corollary \ref{corollary:ShrewdNoMoreRefl} below).  
 In contrast, the following technical result shows that the existence of a  weakly shrewd cardinal that is not shrewd directly implies the existence of reflection points for classes of structures of higher complexities. 
  Moreover, it will also allow us to show that the existence of reflection points for classes of structures of arbitrary complexities has strictly weaker consistency strength than the  existence of a  weakly shrewd cardinal that is not shrewd.

\begin{theorem}\label{theorem:HigherReflection}
 Let $\kappa$ be weakly shrewd cardinal that is not shrewd. 
 \begin{enumerate}
  \item There is a cardinal $\delta>\kappa$ with the property that the set $\{\delta\}$ is definable by a $\Sigma_2$-formula with parameters in $\HH{\kappa}$.  
  
  \item Given $0<n<\omega$ and $\alpha<\kappa$, if $\delta>\kappa$ is a cardinal with the property that the set $\{\delta\}$ is definable by a $\Sigma_2$-formula with parameters in $\HH{\kappa}$, then there exists a cardinal $\alpha<\rho<\delta$ such that $\SRm{\calC}{\rho}$ holds for every class $\calC$ of structures of the same type that is definable by a $\Sigma_n$-formula with parameters in $\HH{\rho}$. 
  
  \item Assume that $0^\#$ does not exist and $\kappa$ is inaccessible. 
  If $\delta>\kappa$ is a cardinal with the property that the set $\{\delta\}$ is definable by a $\Sigma_2$-formula with parameters in $\HH{\kappa}$, 
  then there exists an inaccessible cardinal $\kappa<\varepsilon<\delta$ with the property that, in $\VV_\varepsilon$, the principle $\SRm{\calC}{\kappa}$ holds for every class $\calC$ that is defined by a formula using parameters from $\HH{\kappa}$.\footnote{Note that this conclusion is a statement about the structure $\langle\VV_\varepsilon,\in,\kappa\rangle$ that holds in $\VV$ and is formulated with the help of a formalized satisfaction relation (see, for example, {\cite[Section I.9]{MR750828}}). In particular, this statement also applies to possible classes in $\VV_\varepsilon$ that are defined through formulas with non-standard G\"odel numbers.} 
 \end{enumerate}
\end{theorem}

Note that the set $\{2^{\aleph_0}\}$ is always definable by a $\Sigma_2$-formula without parameters. 
 In particular, the second part of the above theorem tells us that the existence of a weakly shrewd cardinal smaller than  the cardinality of the continuum implies the existence of various local reflection points below $2^{\aleph_0}$. 
  By Theorem \ref{theorem:WShrewNonShrewGCH}, the existence of such cardinals is consistent relative to the existence of a subtle cardinal.

Theorem \ref{theorem:HigherReflection} now allows us to show that $\ZFC$ is consistent with the existence of cardinals with maximal local structural reflection properties. 
 In the light of the results of {\cite[Section 4]{zbMATH06029795}}, the existence of such cardinals can be seen as a localized version of \emph{Vop{\v e}nka's Principle}. 
 Our results show that the consistency strength of this local principle is surprisingly small. 
Moreover, they show that such reflection points can consistently exist below the cardinality of the continuum. 
 As above, we let $\calL_{\dot{c}}$ denote the first-order language extending the language $\calL_\in$ by a constant symbol $\dot{c}$. Given $0<n<\omega$, we let $\mathsf{SR}^-_n$ denote the $\calL_{\dot{c}}$-sentence stating that $\dot{c}$ is an infinite cardinal and  $\SRm{\calC}{\dot{c}}$ holds for every class $\calC$ of structures of the same type that is definable by a $\Sigma_n$-formula in $\calL_\in$ with parameters in $\HH{\dot{c}}$. 
 %

\begin{theorem}\label{theorem:ConsHyper}
 \begin{enumerate}
  \item The   $\calL_\in$-theory $$\ZFC ~ + ~ \anf{\textit{There exists a weakly shrewd cardinal that is not  shrewd}}$$ proves the existence of a transitive model of the $\calL_{\dot{c}}$-theory\footnote{More precisely, given some canonical formalization of this $\calL_{\dot{c}}$-theory, the above $\calL_\in$-theory proves the existence of a transitive set $M$ such that for some $\nu\in M$, every formalized axiom holds in the structure $\langle M,\in,\nu\rangle$ with respect to some formalized satisfaction relation.} $$\ZFC ~ + ~ \Set{\mathsf{SR}^-_n}{0<n<\omega}.$$  
  
  \item The following theories are equiconsistent: 
   \begin{enumerate}
    \item $\ZFC+\anf{\textit{There exists a weakly shrewd cardinal that is not shrewd}}$. 
    
    \item $\ZFC+\Set{\mathsf{SR}^-_n}{0<n<\omega}+\anf{\hspace{2.0pt}\dot{c}<2^{\aleph_0}}$. 
   \end{enumerate}
 \end{enumerate}
\end{theorem}

This result answers the first two questions formulated above. 
The first part  of the above theorem shows that no large cardinal notion with  consistency strength greater than or equal to the existence of a weakly shrewd cardinal that is not shrewd can be characterized through the principle $\mathrm{SR}^-$.  
Moreover, the second part of the corollary shows that no large cardinal property that entails strong inaccessibility can be characterized through the principle $\mathrm{SR}^-$. 
 In particular, this shows that the statement of {\cite[Theorem 3.5]{zbMATH06618251}} 
  needs extra assumptions.\footnote{For example, the 
  argument presented in \cite{zbMATH06618251} 
  works for all cardinals $\kappa$ satisfying $\kappa=\kappa^{{<}\kappa}$.}


In order to answer the third of the above questions, we now turn to the characterizations of large cardinal notions weaker than shrewdness through principles of structural reflection. 
 In the light of the above results, 
 %
 we  introduce further restricted forms of $\Sigma_2$-definability 
 that will enable us to characterize several classical weak large cardinal notions through principles of structural reflection. 
  To motivate the upcoming definition, first observe that for every $0<n<\omega$ and every $\Sigma_n(R)$-definable class $\calC$ of structures, the class of all  isomorphic copies of elements of $\calC$ is again $\Sigma_n(R)$-definable from the same parameters. 
  Next, note that $\Sigma_1$-absoluteness implies that a class $Q$ is definable by a $\Sigma_1$-formula with parameter $z$ if and only if there is a $\Sigma_1$-formula $\varphi(v_0,v_1)$ with the property that for every infinite cardinal $\delta$ with $z\in\HH{\delta^+}$, $$\HH{\delta^+}\cap Q ~ = ~ \Set{x\in\HH{\delta^+}}{\HH{\delta^+}\models\varphi(x,z)}$$ holds. 
 In contrast, let $T$ denote the class of all triples $\langle\delta,x,a\rangle$ with the property that $\delta$ is an infinite cardinal, $x$ is an element of $\HH{\delta^+}$ and $a$ is an element of the set $\mathsf{Fml}$ of formalized $\calL_\in$-formulas with the property that $\mathsf{Sat}(\HH{\delta^+},x,a)$ holds, where $\mathsf{Sat}$ denotes the canonical formalized satisfaction relation for $\calL_\in$-formulas.\footnote{Note that the classes $\mathsf{Fml}$ and $\mathsf{Sat}$ are both defined by $\Sigma_1$-formulas. Moreover, by using codes for negated formulas, it is easy to see that the complement of $\mathsf{Sat}$ is also definable by a $\Sigma_1$-formula.}  
 Then it is easy to see that the class $T$ is definable by a $\Sigma_2$-formula without parameters and \emph{Tarski's Undefinability of Truth Theorem} implies that for every infinite cardinal $\delta$, the intersection $\HH{\delta^+}\cap T$ is not definable in $\HH{\delta^+}$. 
These  observations motivate the restricted form of $\Sigma_2$-definability introduced in the definition below that provides us with a notion of complexity that lies strictly in-between $\Sigma_1$- and $\Sigma_2$-definability (see Proposition \ref{proposition:SpecialCAseSigmanR} below).

\begin{definition}
 Let $R$ be a class and let $n>0$ be a natural number. 
 \begin{enumerate}
  \item Given a set  $z$, a class $S$ is \emph{uniformly locally $\Sigma_n(R)$-definable in the parameter  $z$} if there is a $\Sigma_n(R)$-formula $\varphi(v_0,v_1)$ with the property that $$\HH{\kappa^+}\cap S ~ = ~ \Set{x\in\HH{\kappa^+}}{\langle\HH{\kappa^+},\in,R\rangle\models\varphi(x,z)}$$ holds for every infinite cardinal $\kappa$ with $z\in\HH{\kappa}$. 
  
  \item Given a class $Z$, a class $\calC$ of structures of the same type is a \emph{local $\Sigma_n(R)$-class over $Z$} if the following statements hold: 
   \begin{enumerate}
    \item $\calC$ is closed under isomorphic copies. 
    
    \item $\calC$ is uniformly locally $\Sigma_n(R)$-definable in a parameter contained in $Z$. 
   \end{enumerate}
 \end{enumerate}
\end{definition}

  It can easily be shown that no new large cardinal characterizations can be obtained through canonical $\Pi^1_1$-classes $R$ and the principle $\mathrm{SR}$ for local $\Sigma_1(R)$-classes. 
 First, note that the class $\bar{\calV}$  of all $\calL_\in$-structures that are isomorphic to an element of the class $\calV$ defined above is a local $\Sigma_1(PwSet)$-class over $\emptyset$. 
  This shows that a cardinal $\kappa$ is the least supercompact cardinal if and only if it is the least cardinal with the property that $\SR{\calC}{\kappa}$ holds for every local $\Sigma_1(PwSet)$-class over $\emptyset$. 
  Moreover, if $\VV=\LL$ holds, then the fact that $\HH{\delta^+}=\LL_{\delta^+}$ holds for every infinite cardinal $\delta$ implies that the class $PwSet$ is $\Sigma_1(Cd)$-definable and hence the class $\bar{\calV}$ is definable in the same way. 
  This shows that no $\Pi^1_1$-class $R$ with the property that the class $Cd$ is $\Sigma_1(R)$-definable can be used to characterize large cardinal notions compatible with the assumption $\VV=\LL$ through the principle $\SR{\calC}{\kappa}$ for local $\Sigma_1(R)$-classes.

  In contrast, the next result shows how weak inaccessibility, weak Mahloness and weak $\Pi^1_n$-indescribability, introduced by L{\'e}vy in \cite{MR0281606}, can all be characterized through the validity of the principle $\mathrm{SR}^-$ for certain local $\Sigma_n(R)$-classes. 
   Recall that, given natural numbers $m$ and $n$, a cardinal $\kappa$ is \emph{weakly $\Pi^m_n$-indescribable} if for all relations $A_0,\ldots,A_{m-1}$ on the set $\kappa$ and all $\Pi^m_n$-sentences\footnote{See {\cite[p. 295]{MR1940513}}.} $\Phi$ in $\calL_\in$ that hold in the structure $\langle\kappa,\in,A_0,\ldots,A_{m-1}\rangle$, there exists an ordinal $\lambda<\kappa$ such that $\Phi$ holds in the corresponding substructure $\langle\lambda,\in,A_0,\ldots,A_{m-1}\rangle$ with domain $\lambda$ (see  \cite[Definition 1.(b)]{MR0281606}). 
   Note that a cardinal $\kappa$ is $\Pi^m_n$-indescribable if and only if it is weakly $\Pi^m_n$-indescribable and strongly inaccessible.

\begin{theorem}\label{theorem:CharacterizationSmallLargeCardinals}
 \begin{enumerate}
  \item The following statements are equivalent for every  cardinal $\kappa$: 
   \begin{enumerate}
    \item $\kappa$ is the least weakly inaccessible cardinal. 
  
    \item $\kappa$ is the least cardinal with the property that $\SRm{\calC}{\kappa}$ holds for every local $\Sigma_1(Cd)$-class $\calC$ over $\emptyset$. 
    
        \item $\kappa$ is the least cardinal with the property that $\SRm{\calC}{\kappa}$ holds for every local $\Sigma_1(Cd)$-class $\calC$ over $\HH{\kappa}$. 
    \end{enumerate}

  \item The following statements are equivalent for every  cardinal $\kappa$: 
 \begin{enumerate}
  \item $\kappa$ is the least weakly Mahlo cardinal. 
  
  \item $\kappa$ is the least cardinal with the property that $\SRm{\calC}{\kappa}$ holds for every local $\Sigma_1(Rg)$-class $\calC$ over $\emptyset$. 
  
    \item $\kappa$ is the least cardinal with the property that $\SRm{\calC}{\kappa}$ holds for every local $\Sigma_1(Rg)$-class $\calC$ over $\HH{\kappa}$.
   \end{enumerate}

  \item The following statements are equivalent for every  cardinal $\kappa$ and every $0<n<\omega$: 
 \begin{enumerate}
  \item $\kappa$ is the least weakly $\Pi^1_n$-indescribable cardinal.  
  
    \item $\kappa$ is the least cardinal with the property that $\SRm{\calC}{\kappa}$ holds for every local $\Sigma_{n+1}$-class over $\emptyset$.  
        
   \item $\kappa$ is the least cardinal with the property that $\SRm{\calC}{\kappa}$ holds for every local $\Sigma_{n+1}$-class over $\HH{\kappa}$.  
 \end{enumerate}
 \end{enumerate}
\end{theorem}

 The techniques developed in the proof of the above result will also allow us to show that a large cardinal property isolated by Hamkins is in fact equivalent to L{\'e}vy's notion of weak $\Pi^1_1$-indescribability. 
 Hamkins defined a cardinal $\kappa$ to have the \emph{weakly compact embedding property} if for every transitive set $M$ of cardinality $\kappa$ with $\kappa\in M$, there is a transitive set $N$ and an elementary embedding $\map{j}{M}{N}$ with $\crit{j}=\kappa$ (see \cite{HamkinsTalk}). 
 He then showed that this property implies both weak Mahloness and the tree property. Moreover, he showed that if $\kappa$ is weakly compact and $G$ is $\Add{\omega}{\kappa^+}$-generic over $\VV$,  then $\kappa$ has the weakly compact embedding property in $\VV[G]$. 
  In the proof of Theorem \ref{theorem:CharacterizationSmallLargeCardinals}, we will show that weak $\Pi^1_n$-indescribability can be characterized through the existence of certain elementary embedding and this equivalence also allows us to conclude that weak $\Pi^1_1$-indescribability coincides with the weakly compact embedding property. 
  These observations will also show that the results of {\cite[Section 4]{SECFLC}} only work under the additional assumption that the given cardinal is strongly inaccessible.

 Finally, in unpublished work, Cody, Cox, Hamkins and Johnstone showed that various cardinal invariants of the continuum do not possess the weakly compact embedding property (see \cite{HamkinsTalk}). 
   We will extend these results by showing that various definable cardinals cannot be reflection points of certain classes of structures.  
   For examples, our methods will allow us to show that, although there can consistently exist weakly shrewd cardinals below the \emph{dominating number $\mathfrak{d}$}, the cardinal $\mathfrak{d}$ is neither weakly shrewd nor the successor of a weakly shrewd cardinal.


\section{Shrewd cardinals}

In this section, we derive some consequences of shrewdness that will be used in the proof of Theorem \ref{theorem:ConsStrengthRefl}. 
 The starting point of this analysis is the following embedding characterization for shrewd cardinals that resembles Magidor's classical characterization of supercompactness (see {\cite{MR0295904}} and also {\cite[Theorem 22.10]{MR1994835}}):

\begin{lemma}\label{lemma:ShrewdEmbeddings}
 The following statements are equivalent for every  cardinal $\kappa$: 
 \begin{enumerate}
  \item $\kappa$ is a shrewd cardinal.  
  
  \item For all sufficiently large cardinals $\theta>\kappa$, there exist 
  cardinals $\bar{\kappa}<\bar{\theta}<\kappa$, 
  an elementary submodel $X$\footnote{Note that, in general, the elementary submodel $X$ will not be transitive.} of $\HH{\bar{\theta}}$   %
  and an elementary embedding $\map{j}{X}{\HH{\theta}}$ such that  $\bar{\kappa}+1\subseteq X$,    $j\restriction\bar{\kappa}=\id_{\bar{\kappa}}$ and $j(\bar{\kappa})=\kappa$. 
  
    \item For all cardinals $\theta>\kappa$ and all $z\in\HH{\theta}$, there exist 
  cardinals $\bar{\kappa}<\bar{\theta}<\kappa$, 
  an elementary submodel $X$ of $\HH{\bar{\theta}}$   %
  and an elementary embedding $\map{j}{X}{\HH{\theta}}$ such that  $\bar{\kappa}+1\subseteq X$,    $j\restriction\bar{\kappa}=\id_{\bar{\kappa}}$,  $j(\bar{\kappa})=\kappa$ and $z\in\ran{j}$. 
 \end{enumerate}
\end{lemma}

\begin{proof}
 First, assume that (i) holds. 
  Fix an $\calL_\in$-formula $\Phi(v_0,v_1)$ with the property that $\Phi(A,\delta)$ expresses, in a canonical way, that the conjunction of the following statements holds true: 
  \begin{enumerate}
     \item[(a)] There exist unboundedly many strong limit cardinals. 

   
   \item[(b)] $\delta$ is an inaccessible cardinal. 
   
   \item[(c)] There is a cardinal $\theta>\delta$, a subset $X$ of $\HH{\theta}$ and a bijection $\map{b}{\delta}{X}$ such that the following statements hold: 
    \begin{itemize}    
     \item $\delta+1\subseteq X$, $b(0)=\delta$ and $b(\omega\cdot(1+\alpha))=\alpha$ for all $\alpha<\delta$. 
    
     
     \item The class $\HH{\theta}$ is a set and, given 
     $\alpha_0,\ldots,\alpha_{n-1}<\delta$ and an element $a$ of $\mathsf{Fml}$  that codes a formula with $n$ free variables,  
     we have 
      \begin{equation*}
       \begin{split}
         \langle a,\alpha_0,\ldots,\alpha_{n-1}\rangle\in A ~ & \Longleftrightarrow ~ \mathsf{Sat}(X,\langle b(\alpha_0),\ldots,b(\alpha_{n-1})\rangle,a) \\
          & \Longleftrightarrow ~ \mathsf{Sat}(\HH{\theta},\langle b(\alpha_0),\ldots,b(\alpha_{n-1})\rangle,a). 
         \end{split}
        \end{equation*}
    \end{itemize}
  \end{enumerate}
 
 Fix a cardinal $\theta>\kappa$, $z\in\HH{\theta}$ and a strong limit cardinal $\lambda>\theta$ with the property that $\VV_\lambda$ is sufficiently elementary in $\VV$. 
 Pick an elementary submodel $Y$ of $\HH{\theta}$ of cardinality $\kappa$ with $\kappa\cup\{\kappa,z\}\subseteq Y$ and a bijection $\map{b}{\kappa}{Y}$ satisfying $b(0)=\kappa$, $b(1)=\langle z,\kappa\rangle$ and $b(\omega\cdot(1+\alpha))=\alpha$ for all $\alpha<\kappa$. 
 Define $A$ to be the 
 set of all tuples $\langle a,\alpha_0,\ldots,\alpha_{n-1}\rangle$ with the property that $\alpha_0,\ldots,\alpha_{n-1}<\kappa$,  $a\in\mathsf{Fml}$ codes a formula with $n$ free variables and $\mathsf{Sat}(Y,\langle b(\alpha_0),\ldots,b(\alpha_{n-1})\rangle,a)$ holds.

 Then $\kappa+\lambda=\lambda$ and $\Phi(A,\kappa)$ holds in $\VV_\lambda$. In this situation, the shrewdness of $\kappa$ yields ordinals $\bar{\kappa},\bar{\lambda}<\kappa$ with the property that $\Phi(A\cap\VV_{\bar{\kappa}},\bar{\kappa})$ holds in $\VV_{\bar{\kappa}+\bar{\lambda}}$. 
 By the definition of the formula $\Phi$, we know that $\bar{\kappa}$ is an inaccessible cardinal, $\bar{\lambda}$ is a strong limit cardinal and hence $\bar{\kappa}+\bar{\lambda}=\bar{\lambda}<\kappa$. 
 Moreover, since statements of the form $\anf{x=\HH{\delta}}$ are  absolute between $\VV_{\bar{\lambda}}$ and $\VV$, and the formulas defining the classes $\mathsf{Fml}$ and $\mathsf{Sat}$ are upwards absolute from $\VV_{\bar{\lambda}}$ to $\VV$,  there exists a cardinal $\bar{\kappa}<\bar{\theta}<\bar{\lambda}$, a subset $X$ of $\HH{\bar{\theta}}$ and a bijection $\map{\bar{b}}{\bar{\kappa}}{X}$ such that the following statements hold: 
 \begin{itemize}
  
  \item $\bar{\kappa}+1\subseteq X$, $\bar{b}(0)=\bar{\kappa}$ and $\bar{b}(\omega\cdot(1+\alpha))=\alpha$ for all $\alpha<\bar{\kappa}$. 
  
  
  \item Given an $\calL_\in$-formula $\varphi(v_0,\ldots,v_{n-1})$ and $\alpha_0,\ldots,\alpha_{n-1}<\bar{\kappa}$, we have 
   \begin{equation*}
     \begin{split}
       & \HH{\theta}\models\varphi(b(\alpha_0),\ldots,b(\alpha_{n-1})) ~ \Longleftrightarrow ~ Y\models\varphi(b(\alpha_0),\ldots,b(\alpha_{n-1})) \\
      \Longleftrightarrow ~ & \HH{\bar{\theta}}\models\varphi(\bar{b}(\alpha_0),\ldots,\bar{b}(\alpha_{n-1})) ~ \Longleftrightarrow  ~  X\models\varphi(\bar{b}(\alpha_0),\ldots,\bar{b}(\alpha_{n-1})). 
   \end{split}
  \end{equation*}
 \end{itemize}
 
 This shows that $X$ is an elementary submodel of $\HH{\bar{\theta}}$ with $\bar{\kappa}+1\subseteq X$ and, if we define $$\map{j=b\circ\bar{b}^{{-}1}}{X}{\HH{\theta}},$$ then $j$ is an elementary embedding with $j\restriction\bar{\kappa}=\id_{\bar{\kappa}}$, $j(\bar{\kappa})=\kappa$ and $\langle z,\kappa\rangle\in\ran{j}$. Since we then also have $z\in\ran{j}$, we can conclude that  (iii) holds in this case.

  Now, assume that (ii) holds and assume, towards a contradiction, that there is an $\calL_\in$-formula $\Phi(v_0,v_1)$, an ordinal $\alpha$ and a subset $A$ of $\VV_\kappa$ witnessesing that $\kappa$ is not a shrewd cardinal.  
  Pick a sufficiently large strong limit cardinal $\theta>\kappa+\alpha$ with the property that $\HH{\theta}$ is sufficiently elementary in $\VV$. 
  By our assumption, we can find cardinals $\bar{\kappa}<\bar{\theta}<\kappa$ and  an elementary embedding $\map{j}{X}{\HH{\theta}}$ such that  $\bar{\kappa}+1\subseteq X\prec\HH{\bar{\theta}}$, $j\restriction\bar{\kappa}=\id_{\bar{\kappa}}$ and $j(\bar{\kappa})=\kappa$. 
  
  \begin{claim*}
   $\kappa$ is a strong limit cardinal. 
  \end{claim*}
  
  \begin{proof}[Proof of the Claim]
   Fix an ordinal $\mu<\bar{\kappa}$. Since we know that $\mu\in X$, $j(\mu)=\mu$ and $\HH{\theta}\models\anf{\textit{$\POT{\mu}$ exists}}$, elementarity yields an ordinal $\nu$ in $X$ with $X\models\anf{2^\mu=\nu}$. 
   But then $\HH{\bar{\theta}}\models\anf{2^\mu=\nu}$ and hence $\nu$ is a cardinal in $\VV$ with $2^\mu=\nu<\bar{\theta}<\kappa$.   
   In this situation, elementarity implies that $X\models\anf{2^\mu<\bar{\kappa}}$. Hence, we know that $\bar{\kappa}$ is a strong limit cardinal in $X$ and this allows us to conclude that $\kappa$ is a strong limit cardinal. 
 \end{proof}

  By elementarity, we can now find an ordinal $\alpha$ in $X$ and a subset $A$ of $\VV_{\bar{\kappa}}$ in $X$ with the property that the formula $\Phi(v_0,v_1)$, the ordinal $j(\alpha)$ and the subset $j(A)$ of $\VV_\kappa$ witness that $\kappa$ is not a shrewd cardinal. 
  Since the above claim shows that $\VV_{\bar{\kappa}}\subseteq X$ and $j\restriction\VV_{\bar{\kappa}}=\id_{\VV_{\bar{\kappa}}}$, we know that $j(A)\cap\VV_{\bar{\kappa}}=A$. 
  In particular, it follows that $\Phi(j(A),\kappa)$ holds in $\VV_{\kappa+j(\alpha)}$ and $\Phi(A,\bar{\kappa})$ does not hold in $\VV_{\bar{\kappa}+\alpha}$. 
  Since $\VV_{\bar{\kappa}+\alpha}$ is an element of $X$ with $j(\VV_{\bar{\kappa}+\alpha})=\VV_{\kappa+j(\alpha)}$, we can now use elementarity to derive a contradiction. 
 %
  %
  %
 %
 %
 %
\end{proof}

The above equivalence  allows us to easily deduce several consequences of shrewdness.

\begin{corollary}
 Shrewd cardinals are totally indescribable stationary limits of totally indescribable cardinals. 
\end{corollary}
 
\begin{proof}
 By definition, all shrewd cardinals are totally indescribable. 
 Now, let $\kappa$ be a shrewd cardinal and let $C$ be a closed unbounded subset of $\kappa$. 
  Pick a cardinal $\theta>\beth_\omega(\kappa)$ and use Lemma \ref{lemma:ShrewdEmbeddings} to find cardinals $\bar{\kappa}<\bar{\theta}<\kappa$ and an elementary embedding $\map{j}{X}{\HH{\theta}}$ with $\bar{\kappa}+1\subseteq X\prec\HH{\bar{\theta}}$, $j\restriction\bar{\kappa}=\id_{\bar{\kappa}}$,  $j(\bar{\kappa})=\kappa$ and $C\in\ran{j}$. 
 Then $\bar{\kappa}\in C$ and elementarity implies that $\bar{\theta}>\beth_\omega(\bar{\kappa})$. 
 But this setup ensures that the statement \anf{\emph{$\kappa$ is totally indescribable}} is absolute between $\HH{\theta}$ and $\VV$, and the statement \anf{\emph{$\bar{\kappa}$ is totally indescribable}} is absolute between $\HH{\bar{\theta}}$ and $\VV$. 
  In particular, we can use elementarity to conclude that $\bar{\kappa}\in C$ is totally indescribable.  
 \end{proof}

The next consequence of Lemma \ref{lemma:ShrewdEmbeddings} will be crucial for our characterization of weakly shrewd cardinals that are not shrewd in the next section. 
  This result should be compared with the corresponding statements for supercompact and remarkable cardinals (see {\cite[Proposition 22.3]{MR1994835}} and {\cite[Theorem 1.3]{zbMATH07149985}}). 
Remember that, given a  natural number $n>0$, a cardinal $\kappa$ is \emph{$\Sigma_n$-reflecting} if it is inaccessible and $\VV_\kappa\prec_{\Sigma_n}\VV$ holds.

\begin{corollary}\label{corollary:Refl}
 Shrewd cardinals are $\Sigma_2$-reflecting. 
\end{corollary}

\begin{proof}
 Pick a $\Sigma_2$-formula $\varphi(v_0,\ldots,v_{m-1})$ and sets $z_0,\ldots,z_{m-1}\in\VV_\kappa$ with the property that the statement $\varphi(z_0,\ldots,z_{m-1})$ holds in $\VV$. 
  By $\Sigma_1$-absoluteness, there exists a cardinal $\theta>\kappa$ with the property that $\varphi(z_0,\ldots,z_{m-1})$ holds in $\HH{\theta}$. 
 An application of Lemma \ref{lemma:ShrewdEmbeddings} now yields cardinals $\bar{\kappa}<\bar{\theta}<\kappa$ and an elementary embedding $\map{j}{X}{\HH{\theta}}$ such that  $\bar{\kappa}+1\subseteq X\prec\HH{\bar{\theta}}$,    $j\restriction\bar{\kappa}=\id_{\bar{\kappa}}$, $j(\bar{\kappa})=\kappa$ and $z_0,\ldots,z_{m-1}\in\ran{j}$. 
 Since shrewd cardinals are inaccessible, we have $\VV_{\bar{\kappa}}\subseteq X$ and $j\restriction\VV_{\bar{\kappa}}=\id_{\VV_{\bar{\kappa}}}$. 
   In particular, we know that $z_i\in\VV_{\bar{\kappa}}$ and $j(z_i)=z_i$ holds for all $i<m$. 
 But then $\varphi(z_0,\ldots,z_{m-1})$ holds in $\HH{\bar{\theta}}\subseteq\VV_\kappa$ and hence $\Sigma_1$-absoluteness implies that this statement also holds in $\VV_\kappa$.  
\end{proof}


\section{Weakly shrewd cardinals}

This section contains an analysis of the basic properties of weakly shrewd cardinals. 
We start by slightly modifying the proof of  Lemma \ref{lemma:ShrewdEmbeddings} to obtain an analogous embedding characterization for weakly shrewd cardinals.

\begin{lemma}\label{lemma:WShrewdCharEmb}
 The following statements are equivalent for every cardinal $\kappa$: 
 \begin{enumerate}
  \item $\kappa$ is a weakly shrewd cardinal. 
  
  \item For all sufficiently large cardinals $\theta>\kappa$, there exist   cardinals $\bar{\kappa}<\bar{\theta}$,   an elementary submodel $X$ of $\HH{\bar{\theta}}$    and an elementary embedding $\map{j}{X}{\HH{\theta}}$ with $\bar{\kappa}+1\subseteq X$, $j\restriction\bar{\kappa}=\id_{\bar{\kappa}}$ and $j(\bar{\kappa})=\kappa>\bar{\kappa}$.   
  
    \item For all  cardinals $\theta>\kappa$ and all $z\in\HH{\theta}$, there exist   cardinals $\bar{\kappa}<\bar{\theta}$,   an elementary submodel $X$ of $\HH{\bar{\theta}}$    and an elementary embedding $\map{j}{X}{\HH{\theta}}$ with $\bar{\kappa}+1\subseteq X$, $j\restriction\bar{\kappa}=\id_{\bar{\kappa}}$, $j(\bar{\kappa})=\kappa>\bar{\kappa}$ and $z\in\ran{j}$.   
 \end{enumerate}
\end{lemma}

\begin{proof}
 Assume that (i) holds. 
 Fix a recursive enumeration $\seq{a_l}{l<\omega}$ of the class $\mathsf{Fml}$. 
  Let $\Phi(v_0,v_1)$ be an $\calL_\in$-formula such  that $\Phi(A,\delta)$ expresses that the conjunction of the following statements holds true: 
  \begin{enumerate}   
   \item[(a)] $\delta$ is an infinite cardinal. 
   
   \item[(b)] There is a cardinal $\theta>\delta$, a subset $X$ of $\HH{\theta}$ and a bijection $\map{b}{\delta}{X}$ such that the following statements hold: 
    \begin{itemize}    
     \item $\delta+1\subseteq X$, $b(0)=\delta$ and $b(\omega\cdot(1+\alpha))=\alpha$ for all $\alpha<\delta$. 
    
     
     \item The class $\HH{\theta}$ is a set and, given 
     $\alpha_0,\ldots,\alpha_{n-1}<\delta$ and $l<\omega$ with the property that $a_l$ codes a formula with $n$ free variables,  we have\footnote{We let $\map{\goedel{\cdot}{\ldots,\cdot}}{\On^{n+1}}{\On}$ denote \emph{iterated G\"odel pairing}.} 
      \begin{equation*}
       \begin{split}
         \goedel{l}{\alpha_0,\ldots,\alpha_{n-1}}\in A ~ & \Longleftrightarrow ~ \mathsf{Sat}(X,\langle b(\alpha_0),\ldots,b(\alpha_{n-1})\rangle,a_l) \\
          & \Longleftrightarrow ~ \mathsf{Sat}(\HH{\theta},\langle b(\alpha_0),\ldots,b(\alpha_{n-1})\rangle,a_l).
         \end{split}
        \end{equation*}
    \end{itemize}
  \end{enumerate}

 Fix a cardinal $\theta>\kappa$, $z\in\HH{\theta}$ and a cardinal $\vartheta>2^\theta$ with the property that $\HH{\vartheta}$ is sufficiently elementary in $\VV$. 
 Pick an elementary submodel $Y$ of $\HH{\theta}$ of cardinality $\kappa$ with $\kappa\cup\{\kappa,z\}\subseteq Y$ and a bijection $\map{b}{\kappa}{Y}$ with $b(0)=\kappa$, $b(1)=\langle z,\kappa\rangle$ and $b(\omega\cdot(1+\alpha))=\alpha$ for all $\alpha<\kappa$. 
 Define $A$ to be the 
 set of all ordinals of the form $\goedel{l}{\alpha_0,\ldots,\alpha_{n-1}}$ such that $l<\omega$, $\alpha_0,\ldots,\alpha_{n-1}<\kappa$, $a_l$ codes a formula with $n$ free variables and $\mathsf{Sat}(Y,\langle b(\alpha_0),\ldots,b(\alpha_{n-1})\rangle,a_l)$ holds.   
 Then $\Phi(A,\kappa)$ holds in $\HH{\vartheta}$ and our assumption  yields cardinals $\bar{\kappa}<\bar{\vartheta}$ such that $\bar{\kappa}<\kappa$ and $\Phi(A\cap\bar{\kappa},\bar{\kappa})$ holds in $\HH{\bar{\vartheta}}$. 
 Since the formula defining the predicate $\mathsf{Sat}$ is absolute between $\HH{\vartheta}$, $\HH{\bar{\vartheta}}$ and $\VV$, the definition of $\Phi$ now yields a cardinal $\bar{\kappa}<\bar{\theta}<\bar{\vartheta}$, a subset $X$ of $\HH{\bar{\theta}}$ and a bijection $\map{\bar{b}}{\bar{\kappa}}{X}$ such that the following statements hold:  
\begin{itemize}
  
  \item $\bar{\kappa}+1\subseteq X$, $\bar{b}(0)=\bar{\kappa}$ and $\bar{b}(\omega\cdot(1+\alpha))=\alpha$ for all $\alpha<\bar{\kappa}$. 
  
  
  \item Given an $\calL_\in$-formula $\varphi(v_0,\ldots,v_{n-1})$ and $\alpha_0,\ldots,\alpha_{n-1}<\bar{\kappa}$, we have 
   \begin{equation*}
     \begin{split}
       & \HH{\theta}\models\varphi(b(\alpha_0),\ldots,b(\alpha_{n-1})) ~ \Longleftrightarrow ~ Y\models\varphi(b(\alpha_0),\ldots,b(\alpha_{n-1})) \\
      \Longleftrightarrow ~ & \HH{\bar{\theta}}\models\varphi(\bar{b}(\alpha_0),\ldots,\bar{b}(\alpha_{n-1})) ~ \Longleftrightarrow ~       X\models\varphi(\bar{b}(\alpha_0),\ldots,\bar{b}(\alpha_{n-1})). 
   \end{split}
  \end{equation*}
 \end{itemize}
 
 This shows that $X$ is an elementary submodel of $\HH{\bar{\theta}}$ with $\bar{\kappa}+1\subseteq X$ and the map $\map{j=b\circ\bar{b}^{{-}1}}{X}{\HH{\theta}}$ is an elementary embedding with $j\restriction\bar{\kappa}=\id_{\bar{\kappa}}$, $j(\bar{\kappa})=\kappa$ and $z\in\ran{j}$. These computations show that (iii) holds in this case.

 Now, assume that (ii) holds and (i) fails. Pick an $\calL_\in$-formula $\Phi(v_0,v_1)$ witnessing that $\kappa$ is not a weakly shrewd cardinal, and a sufficiently large cardinal $\vartheta>\kappa$ with the property that $\HH{\vartheta}$ is sufficiently elementary in $\VV$. 
  Then there exists a cardinal $\kappa<\theta<\vartheta$ with the property that for some subset $A$ of $\kappa$, the statement $\Phi(A,\kappa)$ holds in $\HH{\theta}$ and there are no cardinals $\bar{\kappa}<\bar{\theta}$ such that $\bar{\kappa}<\kappa$ and $\Phi(A\cap\bar{\kappa},\bar{\kappa})$ holds in $\HH{\bar{\theta}}$.  
   Let $\theta$ be the minimal cardinal with this property. 
  By our assumption, we can find cardinals $\bar{\kappa}<\bar{\vartheta}$  and an elementary embedding $\map{j}{X}{\HH{\vartheta}}$ with $\bar{\kappa}+1\subseteq X\prec\HH{\bar{\vartheta}}$, $j\restriction\bar{\kappa}=\id_{\bar{\kappa}}$ and $j(\bar{\kappa})=\kappa>\bar{\kappa}$.  
  Since the cardinal $\theta$ is definable in $\HH{\vartheta}$ by an $\calL_\in$-formula with parameter $\kappa$, there is a cardinal $\bar{\theta}$ in $X$ with $j(\bar{\theta})=\theta$. 
  Since $\HH{\vartheta}$ is sufficiently elementary in $\VV$, elementarity now yields  a subset $A$ of $\bar{\kappa}$ in $X$ with the property that $\Phi(j(A),\kappa)$ holds in $\HH{\theta}$ and $\Phi(j(A)\cap\bar{\kappa},\bar{\kappa})$ does not hold in $\HH{\bar{\theta}}$. 
  Since $j(A)\cap\bar{\kappa}=A$, we can use elementarity once more to derive a contradiction and conclude that (i) holds in this case. 
\end{proof}

\begin{corollary}\label{corollary:ShrewdWeaklyShrewd}
 Shrewd cardinals are weakly shrewd. \qed 
\end{corollary}

Building upon the  equivalence established in Lemma \ref{lemma:WShrewdCharEmb}, we now focus on 
consequences of weak shrewdness. 
 The below results will allow us to precisely characterize the class of structural reflecting cardinals that are not shrewd. 
  Moreover, they will allow us to show that the existence of such cardinals below the continuum is consistent.

  We start by proving two basic observation about cardinal arithmetic properties of weakly shrewd cardinals.

\begin{proposition}\label{proposition:WeaklyMahlo}
 Weakly shrewd cardinals are weakly Mahlo. 
\end{proposition}

\begin{proof}
 Let $\kappa$ be a weakly shrewd cardinal. Then we can find a cardinal $\bar{\kappa}<\kappa$, an elementary submodel $X$ of $\HH{\bar{\kappa}^+}$ with $\bar{\kappa}+1\subseteq X$ and an elementary embedding $\map{j}{X}{\HH{\kappa^+}}$ with  $j\restriction\bar{\kappa}=\id_{\bar{\kappa}}$ and $j(\bar{\kappa})=\kappa$. 
 First, assume that $\kappa$ is singular. Then elementarity implies that $\bar{\kappa}$ is singular and there is a cofinal function $\map{c}{\cof{\bar{\kappa}}}{\bar{\kappa}}$ that is an element of $X$. 
 Since $j(\cof{\bar{\kappa}})=\cof{\bar{\kappa}}$, we can use elementarity to conclude that $$j(c)[\cof{\bar{\kappa}}] ~ = ~ c[\cof{\bar{\kappa}}] ~ \subseteq ~ \bar{\kappa}$$ is a cofinal subset of $\kappa$, a contradiction. 
 Now, assume that $\kappa$ is not weakly Mahlo. By elementarity, there exists a closed unbounded subset $C$ of $\bar{\kappa}$ in $X$ that consists of singular ordinals. But then $\bar{\kappa}\in j(C)$ implies that $\bar{\kappa}$ is singular and elementarity implies that $\bar{\kappa}$ is singular in $X$, contradicting the above computations. 
\end{proof}

\begin{proposition}\label{proposition:SRMahlo}
 If $\kappa$ is a weakly shrewd cardinal with $\kappa=\kappa^{{<}\kappa}$, then $\kappa$ is a Mahlo cardinal. 
\end{proposition}

\begin{proof}
 Pick a cardinal $\bar{\kappa}$ and an elementary embedding $\map{j}{X}{\HH{\kappa^+}}$ with the property that  $\bar{\kappa}+1\subseteq X\prec\HH{\bar{\kappa}^+}$, $j\restriction\bar{\kappa}=\id_{\bar{\kappa}}$ and $j(\bar{\kappa})=\kappa>\bar{\kappa}$. 
  Assume, towards a contradiction, that $\kappa$ is not Mahlo. By  Proposition \ref{proposition:WeaklyMahlo}, this implies that $\kappa$ is not a strong limit cardinal and hence our assumptions show that $2^\alpha=\kappa$ holds for some $\alpha<\kappa$. 
  In this situation, elementarity yields an $\alpha<\bar{\kappa}$ with $X\models\anf{2^\alpha=\bar{\kappa}}$. Another application of elementarity now shows that $\bar{\kappa}=2^\alpha=\kappa$ holds, a contradiction.  
\end{proof}

We are now ready to provide the desired characterization of weakly shrewd  cardinals that are not shrewd. 
 This results and its proof should be compared with {\cite[Theorem 1.3]{zbMATH07149985}} that provides an analogous statement for weakly remarkable cardinals that are not remarkable.

 \begin{lemma}\label{lemma:SRnotShrewd}
  The following statements are equivalent for every weakly shrewd cardinal $\kappa$: 
  \begin{enumerate}
   \item $\kappa$ is not a shrewd cardinal. 
   
   \item $\kappa$ is not a $\Sigma_2$-reflecting cardinal. 
   
   \item There exists a cardinal $\delta>\kappa$ with the property that the set $\{\delta\}$ is definable by a $\Sigma_2$-formula with parameters in $\HH{\kappa}$. 
  \end{enumerate}
 \end{lemma}
 
 \begin{proof}
  First, the implication from (ii) to (i) is given by Corollary \ref{corollary:Refl}.  
  
  Next, assume that (iii) holds and $\kappa$ is an inaccessible cardinal.  
  Fix  a $\Sigma_2$-formula $\varphi(v_0,v_1)$, a cardinal $\delta>\kappa$ and $z\in\HH{\kappa}$ with the property that $$\forall x ~ [\varphi(x,z) ~ \longleftrightarrow ~ x=\delta]$$ holds. 
  Then the statement $\exists x ~ \varphi(x,z)$ holds in $\VV$ and, since $\kappa$ is inaccessible, $\Sigma_1$-absoluteness implies that it fails in $\VV_\kappa$. 
  In particular, we know that (ii) holds in this situation. 
  
  Now, assume that (ii) holds. If $\kappa$ is not inaccessible, then Proposition \ref{proposition:SRMahlo} yields an $\alpha<\kappa$ with $2^\alpha>\kappa$ and, since the set $\{2^\alpha\}$ is definable by a $\Sigma_2$-formula with parameter $\alpha$, we can conclude that (iii) holds in this case. 
 We may therefore assume that $\kappa$ is an inaccessible cardinal that is not $\Sigma_2$-reflecting. 
  By standard arguments, this shows that there is an $\calL_\in$-formula $\varphi(v)$ and $z\in\VV_\kappa$ with the property that there is a cardinal $\delta$ such that $z\in\VV_\delta$ and $\varphi(z)$ holds in $\VV_\delta$, and there is no cardinal $\alpha\leq\kappa$ such that $z\in\VV_\alpha$ and $\varphi(z)$ holds in $\VV_\alpha$. 
  Let $\delta$ denote the least cardinal such that  $z\in\VV_\delta$ and $\varphi(z)$ holds in $\VV_\delta$. 
  Then $\delta>\kappa$ and the set $\{\delta\}$ is definable by a $\Sigma_2$-formula with parameter $z$. This shows that (iii) also holds in this case.

  Finally, assume, towards a contradiction, that (i) holds and (ii) fails. 
  By Lemma \ref{lemma:ShrewdEmbeddings}, there is a cardinal $\theta>\kappa$ with the property that for all cardinals $\mu<\nu<\kappa$, there is no elementary embedding $\map{j}{X}{\HH{\theta}}$ with $\mu+1\subseteq X\prec\HH{\nu}$, $j\restriction\mu=\id_\mu$ and $j(\mu)=\kappa$. 
  Let $\theta$ be minimal with this property and pick a strong limit cardinal $\vartheta>\theta$ with the property that $\HH{\vartheta}$ is sufficiently elementary in $\VV$. 
  Using the weak shrewdness of $\kappa$, we find cardinals $\bar{\kappa}<\bar{\theta}<\bar{\vartheta}$  and an elementary embedding $\map{j}{X}{\HH{\vartheta}}$ with $\bar{\kappa}\cup\{\bar{\kappa},\bar{\theta}\}\subseteq X\prec\HH{\bar{\vartheta}}$, $j\restriction\bar{\kappa}=\id_{\bar{\kappa}}$, $j(\bar{\kappa})=\kappa>\bar{\kappa}$ and $j(\bar{\theta})=\theta$.   
  Since $\map{j\restriction(X\cap\HH{\bar{\theta}})}{X\cap\HH{\bar{\theta}}}{\HH{\theta}}$ is an elementary embedding with $\bar{\kappa}+1\subseteq X\cap\HH{\bar{\theta}}\prec\HH{\bar{\theta}}$, $j\restriction\bar{\kappa}=\id_{\bar{\kappa}}$ and $j(\bar{\kappa})=\kappa$, we know that $\bar{\theta}\geq\kappa$ and hence we can conclude that $\bar{\vartheta}>\kappa$.  
  Elementarity now shows that, in $\HH{\bar{\vartheta}}$, there is a cardinal $\rho>\bar{\kappa}$ with the property that, for all cardinals $\mu<\nu<\bar{\kappa}$,  there is no elementary embedding $\map{k}{Y}{\HH{\rho}}$ with $\mu+1\subseteq Y\prec\HH{\nu}$, $k\restriction\mu=\id_\mu$ and $k(\mu)=\bar{\kappa}$. 
  Since this statement can be expressed by a $\Sigma_2$-formula with parameter $\bar{\kappa}\in\VV_\kappa$, $\Sigma_1$-absoluteness implies that it holds in $\VV$ and hence the fact that $\kappa$ is $\Sigma_2$-reflecting causes the statement to also hold in $\VV_\kappa$.    
   Therefore, we can find $\bar{\kappa}<\rho<\kappa$ with the property that, in $\VV_\kappa$,  for all cardinals $\mu<\nu<\bar{\kappa}$,  there is no elementary embedding $\map{k}{Y}{\HH{\rho}}$ with $\mu+1\subseteq Y\prec\HH{\nu}$, $k\restriction\mu=\id_\mu$ and $k(\mu)=\bar{\kappa}$. 
   Since this statement can be expressed in $\VV_\kappa$ by a $\Pi_1$-formula with parameters $\bar{\kappa}$ and $\rho$, $\Sigma_1$-absoluteness implies that it also holds in $\HH{\bar{\vartheta}}$. 
  By the above computations, we have $\bar{\theta}\in X\cap[\kappa,\bar{\vartheta})\neq\emptyset$ and this allows us to define $\lambda=\min(X\cap[\kappa,\bar{\vartheta}))$. 
 We then know that, in $\HH{\bar{\vartheta}}$, there exists a cardinal $\bar{\kappa}<\tau<\lambda$ such that for all cardinals $\mu<\nu<\bar{\kappa}$,  there is no elementary embedding $\map{k}{Y}{\HH{\tau}}$ with $\mu+1\subseteq Y\prec\HH{\nu}$, $k\restriction\mu=\id_\mu$ and $k(\mu)=\bar{\kappa}$. 
 By elementarity, we can find such a cardinal $\tau$ in $X$ and, by the above setup, it follows that $\tau<\kappa$. 
  But then elementarity implies that, in $\HH{\vartheta}$, the cardinal $j(\tau)$ has the property that for all cardinals $\mu<\nu<\kappa$,  there is no elementary embedding $\map{k}{Y}{\HH{j(\tau)}}$ with $\mu+1\subseteq Y\prec\HH{\nu}$, $k\restriction\mu=\id_\mu$ and $k(\mu)=\kappa$. 
  Since $\HH{\vartheta}$ was chosen to be sufficiently elementary in $\VV$, this statement also holds in $\VV$. But this contradicts the minimality of $\theta$, because $\tau<\kappa\leq\bar{\theta}$ implies that $j(\tau)<\theta$.  
  These computations show that (i) implies (ii). 
 \end{proof}

 We now use the above characterization to show that, over the theory $\ZFC$, the existence of a shrewd cardinal is equiconsistent with the existence of an inaccessible weakly shrewd cardinal. 
 The proof of this equiconsistency  will make use of the following notion that will also be central for the proofs of Theorems \ref{theorem:WShrewNonShrewGCH} and \ref{theorem:HigherReflection}.

   \begin{definition}
    Given infinite cardinals $\kappa<\delta$, the cardinal $\kappa$ is  \emph{$\delta$-hyper-shrewd} if for all sufficiently large  cardinals $\theta>\delta$ and all $z\in\HH{\theta}$, there exist 
  cardinals $\bar{\kappa}<\kappa<\delta<\bar{\theta}$, 
  an elementary submodel $X$  of $\HH{\bar{\theta}}$    
   and an elementary embedding $\map{j}{X}{\HH{\theta}}$ with $\bar{\kappa}\cup\{\bar{\kappa},\delta\}\subseteq X$, $j\restriction\bar{\kappa}=\id_{\bar{\kappa}}$,  $j(\bar{\kappa})=\kappa$, $j(\delta)=\delta$ and $z\in\ran{j}$.   
\end{definition}

 Note that the second item in Lemma \ref{lemma:WShrewdCharEmb} implies that all hyper-shrewd cardinals are weakly shrewd. 
 The following lemma provides us with  typical examples of hyper-shrewd cardinals:

 \begin{lemma}\label{lemma:SRnonShrewdFix}
  Let $\kappa$ be a weakly shrewd cardinal that is not shrewd.
  \begin{enumerate}
   \item There exists  $\delta>\kappa$ with the property that the set $\{\delta\}$ is definable by a $\Sigma_2$-formula with parameters in $\HH{\kappa}$.  
   
   \item If $\delta>\kappa$ is a cardinal with the property that the set $\{\delta\}$ is definable by a $\Sigma_2$-formula with parameters in $\HH{\kappa}$, then $\kappa$ is $\delta$-hyper-shrewd. 
  \end{enumerate}
 \end{lemma}
 
 \begin{proof}
   The first statement follows directly from  Lemma \ref{lemma:SRnotShrewd}.
   In the following, fix $y\in\HH{\kappa}$, a cardinal $\delta>\kappa$ and a $\Sigma_2$-formula $\varphi(v_0,v_1)$ with the property that $\delta$ is the unique element satisfying $\varphi(\delta,y)$. 
 Pick a cardinal $\theta>\delta$ with the property that $\varphi(\delta,y)$ holds in $\HH{\theta}$ and fix an element  $z$ of $\HH{\theta}$. 
 Using Lemma \ref{lemma:WShrewdCharEmb}, we find cardinals $\bar{\kappa}<\bar{\theta}$ and an elementary embedding $\map{j}{X}{\HH{\theta}}$ with the property that $\bar{\kappa}+1\subseteq X\prec\HH{\bar{\theta}}$, $j\restriction\bar{\kappa}=\id_{\bar{\kappa}}$, $j(\bar{\kappa})=\kappa>\bar{\kappa}$ and $y,z\in\ran{j}$. 
  Then $j\restriction(\HH{\bar{\kappa}}\cap X)=\id_{\HH{\bar{\kappa}}\cap X}$ 
  %
  and this shows that $y\in X$ and $j(y)=y$.  
  In this situation, elementarity yields a cardinal $\bar{\delta}$ in $X$ such that $j(\bar{\delta})=\delta$ and $\varphi(\bar{\delta},y)$ holds in $\HH{\bar{\theta}}$. 
   %
 But then 
 $\Sigma_1$-absoluteness implies that the statements $\varphi(\delta,y)$ 
  and $\varphi(\bar{\delta},y)$ both hold in $\VV$. 
 By our assumptions on the formula $\varphi$, this allows us to conclude that $\bar{\delta}=\delta=j(\bar{\delta})=j(\delta)$. 
 These computations show that $\kappa$ is $\delta$-hyper-shrewd with respect to $\theta$ and $z$. 
 \end{proof}

 We are now ready to show that shrewd and inaccessible weakly shrewd cardinals possess the same consistency strength.

\begin{lemma}\label{lemma:ConsHyperShrewd}
  If $\kappa$ is an inaccessible cardinal that is $\delta$-hyper-shrewd for some cardinal $\delta>\kappa$, then the interval $(\kappa,\delta)$ contains an inaccessible cardinal and, if $\varepsilon$ is the least inaccessible cardinal above $\kappa$, then $\kappa$ is a shrewd cardinal in $\VV_\varepsilon$.  
\end{lemma}

\begin{proof}
 Pick a sufficiently large cardinal $\vartheta$, cardinals $\bar{\kappa}<\kappa<\delta<\bar{\vartheta}$ and an elementary embedding $\map{j}{X}{\HH{\vartheta}}$ with $\bar{\kappa}\cup\{\bar{\kappa},\delta\}\subseteq X\prec\HH{\bar{\vartheta}}$, $j\restriction\bar{\kappa}=\id_{\bar{\kappa}}$,  $j(\bar{\kappa})=\kappa$ and $j(\delta)=\delta$.   
 %
 %
 By elementarity, there is exists an inaccessible cardinal in the interval $(\bar{\kappa},\delta)$ that is an element of $X$. 
 This directly implies that there exists an inaccessible cardinal in the interval $(\kappa,\delta)$.

  Now, let $\varepsilon$ denote the least inaccessible cardinal above $\kappa$ and assume, towards a contradiction, that $\kappa$ is not a shrewd cardinal in $\VV_\varepsilon$.  
   By Lemma \ref{lemma:ShrewdEmbeddings}, there exists a cardinal $\kappa<\theta<\varepsilon$ such that for all cardinals $\mu<\nu<\kappa$, there is no elementary embedding  $\map{k}{Y}{\HH{\theta}}$ in $\VV_\varepsilon$ such that $\bar{\kappa}+1\subseteq Y\prec\HH{\nu}$, $k\restriction\mu=\id_\mu$ and $k(\mu)=\kappa$. 
  Let $\theta$ denote the least cardinal with this property. 
  Since  $\varepsilon$ and $\theta$ are both definable in $\HH{\vartheta}$ by $\calL_\in$-formulas that only use the parameter $\kappa$, we can find cardinals $\bar{\varepsilon}$ and $\bar{\theta}$ in $X$ with $j(\bar{\varepsilon})=\varepsilon$ and $j(\bar{\theta})=\theta$. 
  Then $\bar{\varepsilon}$ is the least inaccessible cardinal above $\bar{\kappa}$ and, since Proposition \ref{proposition:WeaklyMahlo} shows that $\kappa$ is a Mahlo cardinal, it follows that  $\bar{\theta}<\bar{\varepsilon}<\kappa$. 
 But this yields a contradiction, because the map $\map{j\restriction(\HH{\bar{\theta}}\cap X)}{\HH{\bar{\theta}}\cap X}{\HH{\theta}}$ is an elementary embedding with $\bar{\kappa}+1\subseteq\HH{\bar{\theta}}\cap X\prec\HH{\bar{\theta}}$, $j\restriction\bar{\kappa}=\id_{\bar{\kappa}}$ and $j(\bar{\kappa})=\kappa>\bar{\kappa}$, and the inaccessibility of $\varepsilon$ implies that this map is an element of $\VV_\varepsilon$.  
\end{proof}

We now use a small variation of a standard argument, commonly attributed to Kunen, to show that both weak shrewdness and hyper-shrewdness are  downwards absolute to $\LL$.

\begin{proposition}\label{proposition:DownwardsToL}
 Let $M$ be an inner model, let $\kappa<\vartheta$ and $\bar{\kappa}<\bar{\vartheta}$ be $M$-cardinals, and let   $\map{j}{Y}{\HH{\vartheta}^M}$ be an elementary embedding with $\bar{\kappa}+1\subseteq Y\prec\HH{\bar{\vartheta}}^M$, $j\restriction\bar{\kappa}=\id_{\bar{\kappa}}$ and $j(\bar{\kappa})=\kappa>\bar{\kappa}$.
   If $\bar{\kappa}<\bar{\theta}<\bar{\vartheta}$ is an $M$-cardinal in $Y$ with $\HH{\bar{\theta}}^M\in\HH{\bar{\vartheta}}^M$,  
   then there exists an elementary submodel $X\in Y$ of  $\HH{\bar{\theta}}^M$ with $\bar{\kappa}+1\subseteq X\subseteq Y$ and the property that the embedding $j\restriction X$ is an element of $M$. 
\end{proposition}

\begin{proof}
 Our assumptions ensure that $Y$ contains an elementary submodel $X$ of $\HH{\bar{\theta}}^M$ of cardinality $\bar{\kappa}$ with $\bar{\kappa}+1\subseteq X$. 
 By elementarity, there exists a bijection $\map{b}{\bar{\kappa}}{X}$ in $Y$. Since we have $\bar{\kappa}\cup\{b\}\subseteq Y\prec\HH{\bar{\vartheta}}^M$, we know that $X$ is a subset of $Y$. 
 Given $x\in X$, we have $j(b^{{-}1}(x))=b^{{-}1}(x)$ and hence we know that $j(x)=(j(b)\circ b^{{-}1})(x)$ holds. Since $b\in\HH{\bar{\vartheta}}^M$ and $j(b)\in\HH{\vartheta}^M$, this shows that $j\restriction X$ is an element of $M$.   
\end{proof}

\begin{corollary}\label{corollary:Downwards}
 \begin{enumerate}
  \item Every weakly shrewd cardinal is a weakly shrewd cardinal in $\LL$. 
  
  \item Given a cardinal $\delta$, every $\delta$-hyper-shrewd cardinal is a $\delta$-hyper-shrewd cardinal in $\LL$. 
 \end{enumerate}
\end{corollary}

\begin{proof}
 (i) Assume that $\kappa$ is a weakly shrewd cardinal. Fix an $\LL$-cardinal $\theta>\kappa$ and a regular cardinal $\vartheta>\theta$. 
    By Lemma \ref{lemma:WShrewdCharEmb}, there exists cardinals $\bar{\kappa}<\bar{\vartheta}$ and an elementary embedding  $\map{k}{Y}{\HH{\vartheta}}$ with $\bar{\kappa}+1\subseteq Y\prec\HH{\vartheta}$, $k\restriction\bar{\kappa}=\id_{\bar{\kappa}}$, $k(\bar{\kappa})=\kappa>\bar{\kappa}$ and $\theta\in\ran{k}$. 
    Pick $\bar{\theta}\in Y$ with $k(\bar{\theta})=\theta$. 
    Since $\bar{\theta}\in \LL_{\bar{\vartheta}}\cap Y\prec\LL_{\bar{\vartheta}}$, we can apply Proposition \ref{proposition:DownwardsToL} to find an elementary submodel $X\in\LL_{\bar{\vartheta}}\cap Y$ of $\LL_{\bar{\theta}}$ with $\betrag{X}^\LL=\bar{\kappa}$, $\bar{\kappa}+1\subseteq X$ and $j=k\restriction X\in\LL$. 
    But then $\map{j}{X}{\LL_\theta}$ is an elementary embedding in $\LL$ with $j\restriction\bar{\kappa}=\id_{\bar{\kappa}}$ and $j(\bar{\kappa})=\kappa>\bar{\kappa}$. 
 These computations show that $\kappa$ is weakly shrewd in $\LL$. 
 
 (ii) Now, assume that $\kappa$ is $\delta$-hyper-shrewd for some cardinal $\delta>\kappa$. Fix an $\LL$-cardinal $\theta>\delta$, $z\in\LL_\theta$ and a sufficiently large regular cardinal $\vartheta>\theta$. 
  Pick cardinals $\bar{\kappa}<\bar{\vartheta}$ and an elementary embedding  $\map{k}{Y}{\HH{\vartheta}}$ with $\bar{\kappa}\cup\{\bar{\kappa},\delta\}\subseteq Y\prec\HH{\vartheta}$, $k\restriction\bar{\kappa}=\id_{\bar{\kappa}}$, $k(\bar{\kappa})=\kappa>\bar{\kappa}$, $k(\delta)=\delta$ and $\theta,z\in\ran{k}$. 
 In addition, fix $\bar{\theta},\bar{z}\in Y$ with $j(\bar{\theta})=\theta$ and $j(\bar{z})=z$. 
 Now, use Proposition \ref{proposition:DownwardsToL} to find an elementary submodel $X\in\LL_{\bar{\vartheta}}\cap Y$ of $\LL_{\bar{\theta}}$  with $\betrag{X}^\LL=\bar{\kappa}$, $\bar{\kappa}\cup\{\bar{\kappa},\delta,\bar{z}\}\subseteq X$ and $j=k\restriction X\in\LL$. 
 Then $j$ witnesses that $\kappa$ is $\delta$-hyper-shrewd with respect to $\theta$ and $z$ in $\LL$. 
  In particular, we have shown that $\kappa$ is a $\delta$-hyper-shrewd cardinal in $\LL$. 
\end{proof}

\begin{corollary}\label{corollary:EquiInacc}
 The following statements are equiconsistent over $\ZFC$: 
 \begin{enumerate}
  \item There exists a shrewd cardinal. 
  
  \item There exists an inaccessible weakly shrewd cardinal. 
  
    \item There exists a weakly shrewd cardinal. 
 \end{enumerate}
\end{corollary}

\begin{proof}
 Let $\kappa$ be a weakly shrewd cardinal. Then Corollary \ref{corollary:Downwards} shows that $\kappa$ is a weakly shrewd cardinal in $\LL$ and Proposition \ref{proposition:SRMahlo} implies that $\kappa$ is inaccessible in $\LL$. If $\kappa$ is not a shrewd cardinal in $\LL$, then a combination of Lemma \ref{lemma:SRnonShrewdFix} with Lemma \ref{lemma:ConsHyperShrewd} yields an ordinal $\varepsilon>\kappa$ that is inaccessible in $\LL$ and has the property that $\kappa$ is a shrewd cardinal in $\LL_\varepsilon$. This shows that, over $\ZFC$, the consistency of (iii) implies the consistency of (i). Since all shrewd cardinals are inaccessible weakly shrewd cardinals, this implication yields the statement of the corollary.  
\end{proof}

We now show that the existence of weakly shrewd cardinals that are not shrewd is consistent by proving that subtle cardinals provide a proper upper bound for the consistency strength of the existence of hyper-shrewd cardinals. 
Remember that a cardinal $\delta$ is \emph{subtle} if for every sequence $\seq{d_\alpha}{\alpha<\delta}$ with $d_\alpha\subseteq\alpha$ for all $\alpha<\delta$ and every closed unbounded subset $C$ of $\delta$, there exist $\alpha,\beta\in C$ with $\alpha<\beta$ and $d_\alpha=d_\beta\cap\alpha$ (see \cite{JensenKunen1969:Ineffable}).

\begin{lemma}\label{lemma:SubtleHyperShrewd}
 If $\delta$ is a subtle cardinal, then the set of all inaccessible $\delta$-hyper-shrewd cardinals is stationary in $\delta$.   
\end{lemma}

\begin{proof}
 Let $\delta$ be a subtle cardinal and assume, towards a contradiction, that the set of inaccessible $\delta$-hyper-shrewd cardinals is not stationary in $\delta$.   
 Since $\delta$ is an inaccessible limit of inaccessible cardinals, there is a closed unbounded subset $C$ of $\delta$ consisting of cardinals that are not $\delta$-hyper-shrew and are limits of inaccessible cardinals. 
Given an element $\kappa$ of $C$ that is inaccessible, our assumptions yield a cardinal $\theta_\kappa>\delta$ and an element $z_\kappa$ of $\HH{\theta_\kappa}$ with the property that for all cardinals $\bar{\kappa}<\kappa<\delta<\bar{\theta}$ there is no elementary embedding $\map{j}{X}{\HH{\theta_\kappa}}$ satisfying $\bar{\kappa}\cup\{\bar{\kappa},\delta\}\subseteq X\prec\HH{\bar{\theta}}$, $j\restriction\bar{\kappa}=\id_{\bar{\kappa}}$,  $j(\bar{\kappa})=\kappa$, $j(\delta)=\delta$ and $z_\kappa\in\ran{j}$.   
 For each inaccessible cardinal $\kappa$ in $C$, fix an elementary submodel $X_\kappa$ of $\HH{\theta_\kappa}$ of cardinality $\kappa$ with $\kappa\cup\{\kappa,\delta,z_\kappa\}\subseteq X_\kappa$, and a bijection $\map{b_\kappa}{\kappa}{X_\kappa}$ with $b(0)=\kappa$, $b(1)=\delta$, $b(2)=\langle z_\kappa,\kappa\rangle$ and $b(\omega\cdot(1+\alpha))=\alpha$ for all $\alpha<\kappa$. 
   Next, if $\lambda$ is an element of $C$ that is not inaccessible, then $\lambda$ is a singular cardinal and we fix a strictly increasing cofinal function $\map{c_\lambda}{\cof{\lambda}}{\lambda}$ with $c_\lambda(0)=0$.

   Fix an enumeration $\seq{a_l}{l<\omega}$ of $\mathsf{Fml}$ and pick a sequence $\seq{d_\alpha}{\alpha<\delta}$ such that the following statements hold for all $\alpha<\kappa$: 
   \begin{itemize}
    \item $d_\alpha$ is a subset of $\alpha$. 
    
    \item If $\alpha$ is an inaccessible cardinal in $C$, then the set $d_\alpha$ consists of all ordinals of the form $\goedel{l}{\alpha_0,\ldots,\alpha_{n-1}}$ such that $l<\omega$, $\alpha_0,\ldots,\alpha_{n-1}<\alpha$, $a_l$ codes a formula with $n$ free variables  and $\mathsf{Sat}(X_\alpha,\langle b(\alpha_0),\ldots,b(\alpha_{n-1})\rangle,a_l)$ holds. 
    
    \item If $\alpha$ is a singular cardinal in $C$, then the set $d_\alpha$ consists of all ordinals of the form $\goedel{\cof{\alpha}}{c_\alpha(\xi)}$ with $\xi<\cof{\alpha}$.  
   \end{itemize}

  The subtlety of $\delta$ now yields $\alpha,\beta\in C$ with $\alpha<\beta$ and $d_\alpha=d_\beta\cap\alpha$.

  \begin{claim*}
   The ordinals $\alpha$ and $\beta$ are both  inaccessible. 
  \end{claim*} 
  
  \begin{proof}[Proof of the Claim]
   Assume, towards a contradcition, that either $\alpha$ or $\beta$ is not inaccessible. 
   If $\alpha$ is not inaccessible, then $\goedel{\cof{\alpha}}{0}\in d_\alpha\subseteq d_\beta$ and hence $\cof{\alpha}=\cof{\beta}<\alpha<\beta$. 
   In the other case, if $\beta$ is not inaccessible, then the fact that $d_\beta\cap\alpha=d_\alpha\neq\emptyset$ yields a $\xi<\cof{\beta}$ with $\goedel{\cof{\beta}}{c_\beta(\xi)}\in d_\alpha$ and this shows that $\cof{\alpha}=\cof{\beta}<\alpha<\beta$ also holds in this case. 
  Let $\zeta_0$ be the minimal element of $\cof{\alpha}$ with $c_\beta(\zeta_0)\geq\alpha$. Then there exists $\zeta_0<\zeta_1<\cof{\alpha}$ such that $c_\beta(\xi)<c_\alpha(\zeta_1)$ holds for all $\xi<\zeta_0$. 
  Since our setup ensures that $\goedel{\cof{\alpha}}{c_\alpha(\zeta_1)}\in d_\alpha\subseteq d_\beta$, there exists $\xi<\cof{\alpha}$ with $c_\alpha(\zeta_1)=c_\beta(\xi)$. 
  Moreover, since $c_\beta$ is strictly increasing, the fact that $c_\beta(\xi)=c_\alpha(\zeta_1)<\alpha\leq c_\beta(\zeta_0)$ implies that $\xi<\zeta_0$ and hence we can conclude  that $c_\beta(\xi)<c_\alpha(\zeta_1)=c_\beta(\xi)$, a contradiction. 
  \end{proof}

 By the definition of the sequence $\seq{d_\alpha}{\alpha<\delta}$, we now know that $$X_\alpha\models\varphi(b_\alpha(\alpha_0),\ldots,b_\alpha(\alpha_{n-1})) ~ \Longleftrightarrow ~ \HH{\theta_\beta}\models\varphi(b_\beta(\alpha_0),\ldots,b_\beta(\alpha_{n-1}))$$ holds for every $\calL_\in$-formula $\varphi(v_0,\ldots,v_{n-1})$ and all $\alpha_0,\ldots,\alpha_{n-1}<\alpha$. 
  In particular, if we define $$\map{j=b_\beta\circ b_\alpha^{{-}1}}{X_\alpha}{\HH{\theta_\beta}},$$ then $j$ is an elementary embedding with $\alpha\cup\{\alpha,\delta\}\subseteq X_\alpha\prec\HH{\theta_\alpha}$, $j\restriction\alpha=\id_\alpha$,  $j(\alpha)=\beta$, $j(\delta)=\delta$ and $\langle z_\beta,\beta\rangle\in\ran{j}$. But then elementarity ensures that $z_\beta$ is also an element of $\ran{j}$, contradicting the above assumptions.    
\end{proof}

We continue by showing that weakly shrewd cardinals can exist below the cardinality of the continuum:

\begin{lemma}\label{lemma:AddGenExtHyperShrew}
 If $\kappa$ is a cardinal that is $\delta$-hyper-shrewd for some cardinal $\delta>\kappa$ and $G$ is $\Add{\omega}{\delta}$-generic over $\VV$, then $\kappa$ is $\delta$-hyper-shrewd in $\VV[G]$. 
\end{lemma}

\begin{proof}
 Work in $\VV$ and pick a sufficiently large cardinal $\theta>2^\delta$ and an $\Add{\omega}{\delta}$-name $\dot{z}$ in $\HH{\theta}$. 
  By our assumptions, there are cardinals $\bar{\kappa}<\kappa<\delta<\bar{\theta}$    and an elementary embedding $\map{k}{Y}{\HH{\theta}}$ with $\bar{\kappa}\cup\{\bar{\kappa},\delta\}\subseteq Y\prec\HH{\bar{\theta}}$, $k\restriction\bar{\kappa}=\id_{\bar{\kappa}}$,  $k(\bar{\kappa})=\kappa$, $k(\delta)=\delta$ and $\dot{z}\in\ran{k}$.   
  Let $X$ be an elementary submodel of $Y$ of cardinality $\bar{\kappa}$ with the property that $\bar{\kappa}\cup\{\bar{\kappa},\delta\}\subseteq X$ and $\dot{z}\in k[X]$. 
   Then the map $\map{j=k\restriction X}{X}{\HH{\theta}}$ is an elementary embedding with $j\restriction\bar{\kappa}=\id_{\bar{\kappa}}$,  $j(\bar{\kappa})=\kappa$, $j(\delta)=\delta$ and $\dot{z}\in\ran{j}$. 
  Since $\betrag{X}=\bar{\kappa}<\delta$, there exists a permutation $\sigma$ of $\delta$ that extends the injection $j\restriction(X\cap\delta)$. 
 Let $\tau$ denote the automorphism of the partial order $\Add{\omega}{\delta}$ that is induced by the action of $\sigma$ on the supports of conditions, i.e. given a condition $p$ in $\Add{\omega}{\delta}$, we have $\supp{\tau(p)}=\sigma[\supp{p}]$ and $\tau(p)(\sigma(\alpha))=p(\alpha)$ for all $\alpha\in\supp{p}$. 
  Then it is easy to see that $\tau(p)=j(p)$ holds for every condition in $\Add{\omega}{\delta}$ that is an element of $X$. 
  Moreover, since elementarity implies that $2^\delta<\bar{\theta}$, we also know that the automorphism $\tau$ is an element of $\HH{\bar{\theta}}$.

  Now, work in $\VV[G]$ and set $\bar{G}=\tau^{{-}1}[G]$. 
  Then $j[\bar{G}\cap X]\subseteq G$ and we can therefore construct a canonical lifting $$\Map{j_*}{X[\bar{G}\cap X]}{\HH{\theta}^\VV[G]}{\dot{x}^{\bar{G}\cap X}}{j(\dot{x})^G}$$ of $j$. 
  Then $\HH{\theta}^\VV[G]=\HH{\theta}^{\VV[G]}$ and, since $\tau\in\HH{\bar{\theta}}^\VV$, we also have $$\HH{\bar{\theta}}^\VV[G] ~ = ~ \HH{\bar{\theta}}^\VV[\bar{G}] ~ = ~ \HH{\bar{\theta}}^{\VV[G]}.$$ 
  In particular, we know that the set $X[\bar{G}\cap X]$ is an  elementary submodels of $\HH{\bar{\theta}}^{\VV[G]}$ and we can conclude that the embedding $j_*$ witnesses the $\delta$-hyper-shrewdness of $\kappa$ with respect to $\theta$ and $\dot{z}^G$ in $\VV[G]$. 
\end{proof}

The next lemma will allow us to show that the existence of a hyper-shrewd cardinal is equiconsistent to the existence of a weakly shrewd cardinal that is not shrewd.

\begin{lemma}\label{lemma:MakeDefShrewd}
 Assume that $\VV=\LL$. If $\kappa$ is a cardinal that is $\delta$-hyper-shrewd for some cardinal $\delta>\kappa$ and $G$ is $\Add{\delta^+}{1}$-generic over $\VV$, then, in $\VV[G]$, the set $\{\delta\}$ is definable by a $\Sigma_2$-formula without parameters and the cardinal $\kappa$ is inaccessible, weakly shrewd and not shrewd. 
\end{lemma}

\begin{proof}
 Work in $\VV$ and fix an $\calL_\in$-formula $\varphi(v_0,v_1)$, a cardinal $\theta>\kappa$ and a subset $A$ of $\kappa$ with the property that the statement $\varphi(A,\kappa)$ holds in $\HH{\theta}^{\VV[G]}$. 
   Pick a sufficiently large cardinal $\vartheta>\max\{\delta^+,\theta\}$, cardinals $\bar{\kappa}<\kappa<\delta^+<\bar{\vartheta}$,    
   and an elementary embedding $\map{j}{X}{\HH{\vartheta^+}}$ with the property that $\bar{\kappa}\cup\{\bar{\kappa},\delta,\bar{\vartheta}\}\subseteq X\prec\HH{\bar{\vartheta}^+}$, $j\restriction\bar{\kappa}=\id_{\bar{\kappa}}$,  $j(\bar{\kappa})=\kappa$, $j(\delta)=\delta$ and $A\in\ran{j}$.   
   Since Proposition \ref{proposition:SRMahlo} implies that $\kappa$ is an inaccessible cardinal, we know that $\bar{\kappa}$ is also inaccessible, $\VV_{\bar{\kappa}}$ is a subset of $X$ and $j\restriction\VV_{\bar{\kappa}}=\id_{\VV_{\bar{\kappa}}}$. 
   In particular, the fact that $A$ is an element of $\ran{j}$ implies that $A\cap\bar{\kappa}\in X$ and $j(A\cap\bar{\kappa})=A$. 
   Using the weak homogeneity of $\Add{\delta^+}{1}$ and the fact that $\HH{\vartheta^+}^\VV[G]=\HH{\vartheta^+}^{\VV[G]}$, our assumptions now imply that, in $\HH{\vartheta^+}$, every condition in the partial order $\Add{\delta^+}{1}$ forces the statement $\varphi(A,\kappa)$ to hold in the $\HH{\theta}$ of the generic extension of $\HH{\vartheta^+}$. 
   Pick a cardinal $\bar{\theta}$ in $X$ with $j(\bar{\theta})=\theta$. 
  Since  $\delta$ is a fixed point of $j$, elementarity implies that, in $\HH{\bar{\vartheta}^+}$, every condition in $\Add{\delta^+}{1}$ forces the statement $\varphi(A\cap\bar{\kappa},\bar{\kappa})$ to hold in the $\HH{\bar{\theta}}$ of the generic extension of $\HH{\bar{\vartheta}^+}$. 
 
  Next, observe that we have $\HH{\bar{\vartheta}^+}^\VV[G]=\HH{\bar{\vartheta}^+}^{\VV[G]}$ and therefore the fact that $\bar{\theta}^+<\bar{\vartheta}^+$ holds in $\VV$ implies that $$\HH{\bar{\theta}}^{\HH{\bar{\vartheta}^+}^{\VV[G]}} ~ = ~ \HH{\bar{\theta}}^{\VV[G]}.$$
  This allows us to conclude that $\varphi(A\cap\bar{\kappa},\bar{\kappa})$ holds in $\HH{\bar{\theta}}^{\VV[G]}$. 
  Using the fact that, in $\VV$, the partial order $\Add{\delta^+}{1}$ is ${<}\delta^+$-closed and satisfies the $\delta^{++}$-chain condition, the above computations allow us conclude that $\kappa$ is weakly shrewd in $\VV[G]$.  

 Now, work in $\VV[G]$. Then $\delta^+$ is the least ordinal containing a non-constructible subset and hence the set $\{\delta\}$ is definable by a $\Sigma_2$-formula without parameters.  
 By Lemma \ref{lemma:SRnotShrewd}, this allows us to conclude that $\kappa$ is an inaccessible weakly shrewd cardinal that is not a shrewd cardinal. 
\end{proof}

We are now ready to determine the position of accessible weakly shrewd cardinals in the large cardinal hierarchy.

\begin{proof}[Proof of Theorem \ref{theorem:WShrewNonShrewGCH}]
 (i) Let $\kappa$ be a weakly shrewd cardinal that is not shrewd. 
  Then Lemma \ref{lemma:SRnonShrewdFix} shows that $\kappa$ is $\delta$-hyper-shrewd for some cardinal $\delta>\kappa$ and Corollary \ref{corollary:Downwards} implies that $\kappa$ is a $\delta$-hyper-shrewd cardinal in $\LL$. 
  Since Proposition \ref{proposition:SRMahlo} ensures that $\kappa$ is an inaccessible cardinal in $\LL$, Lemma \ref{lemma:ConsHyperShrewd} now allows us to find an ordinal $\varepsilon>\kappa$ that is inaccessible in $\LL$ and has the property that $\kappa$ is a shrewd cardinal in $\LL_\varepsilon$. 

 (ii) Let $\delta$ be the least subtle cardinal and let $C$ be a closed unbounded subset of $\delta$. By Lemma  \ref{lemma:SubtleHyperShrewd}, there is an inaccessible weakly shrewd cardinal $\kappa$ in $C$. 
  Since the statement \anf{$\mu$ is subtle} is absolute between $\HH{\mu^+}$ and $\VV$ for every infinite cardinal $\mu$, the minimality of $\delta$ implies that the set $\{\delta\}$ is definable by a $\Sigma_2$-formula without parameters. 
  An application of Lemma \ref{lemma:SRnotShrewd} now allows us to conclude that $\kappa$ is not a shrewd cardinal.

 (iii) First, assume that $\kappa$ is a weakly shrewd cardinal that is not  shrewd. 
 Using Lemma \ref{lemma:SRnonShrewdFix}, we find a cardinal $\delta>\kappa$ such that $\kappa$ is $\delta$-hyper-shrewd. 
 Let $G$ be $\Add{\omega}{\delta}$-generic over $\VV$. Then Lemma \ref{lemma:AddGenExtHyperShrew} implies that $\kappa$ is a weakly shrewd cardinal smaller than $2^{\aleph_0}$ in $\VV[G]$. 
 These arguments show that, over the theory $\ZFC$,  the consistency of the statement (a) listed in the theorem implies the consistency of the statement (c) listed in the theorem. 
 
 Now,  assume that $\kappa$ is a weakly shrewd cardinal that is not inaccessible. Then $\kappa$ is not a shrewd cardinal and hence Lemma \ref{lemma:SRnonShrewdFix} shows that $\kappa$ is $\delta$-hyper-shrewd for some cardinal $\delta>\kappa$. 
  In this situation, Corollary \ref{corollary:Downwards} implies that $\kappa$ is a $\delta$-hyper-shrewd cardinal in $\LL$. 
 Let $G$ be $\Add{\delta^+}{1}^\LL$-generic over $\LL$.
  Then Lemma \ref{lemma:MakeDefShrewd} shows that, in $\LL[G]$, the cardinal $\kappa$ inaccessible, weakly shrewd and not shrewd. 
 These arguments show that the consistency of the statement (b) listed in the theorem implies the consistency of the statement (a) listed in the theorem. 
\end{proof}

We end this section by showing that weak shrewdness is also a direct consequence of large cardinal notions introduced by Schindler in \cite{schindler_2001} and Wilson in \cite{zbMATH07149985}. 
Note that the results of \cite{MR3598793} show that the below definition of remarkability is equivalent to Schindler's original definition.

\begin{definition}
 \begin{enumerate}
  \item (Schindler) A cardinal $\kappa$ is \emph{remarkable} if for every ordinal $\alpha>\kappa$, there is an ordinal $\beta<\kappa$ and a generic elementary embedding $\map{j}{\VV_\beta}{\VV_\alpha}$ with $j(\crit{j})=\kappa$. 

  \item (Wilson)
 A cardinal $\kappa$ is \emph{weakly remarkable} if for every ordinal $\alpha>\kappa$, there is an ordinal $\beta$ and a generic elementary embedding $\map{j}{\VV_\beta}{\VV_\alpha}$ with $j(\crit{j})=\kappa$. 
 \end{enumerate}
\end{definition}

 Results of Wilson in \cite{zbMATH07149985} now allow us to find additional natural examples of hyper-shrewd cardinals.

\begin{lemma}\label{lemma:WeaklyRemarkableShrewd}
 Every weakly remarkable cardinal is weakly shrewd. 
\end{lemma}

\begin{proof}
  Let $\kappa$ be a weakly remarkable cardinal. Assume, towards a contradiction, that $\kappa$ is not weakly shrewd and let $\theta>\kappa$ denote the least cardinal with the property that for all cardinals $\bar{\kappa}<\bar{\theta}$,   there is no elementary embedding $\map{j}{X}{\HH{\theta}}$ with $\bar{\kappa}+1\subseteq X\prec\HH{\theta}$, $j\restriction\bar{\kappa}=\id_{\bar{\kappa}}$ and $j(\bar{\kappa})=\kappa>\bar{\kappa}$.   
 Pick a strong limit cardinal $\vartheta>\theta$ with the property that $\HH{\vartheta}=\VV_\vartheta$ is sufficiently elementary in $\VV$. 
 We then know that $\theta$ is definable in $\VV_\vartheta$ by a formula with parameter $\kappa$. 
 By our assumptions, we can now find  ordinals $\bar{\kappa}<\bar{\vartheta}$ and a generic elementary embedding $\map{j}{\VV_{\bar{\vartheta}}}{\VV_\vartheta}$ with $\crit{j}=\bar{\kappa}$ and $j(\bar{\kappa})=\kappa$. 
  Then elementarity implies that $\bar{\kappa}$ is a cardinal in $\VV$  and $\bar{\vartheta}$ is a strong limit cardinal in $\VV$. 
  Moreover, the definability of $\theta$ yields a $\VV$-cardinal $\bar{\theta}$ with $j(\bar{\theta})=\theta$. 
  An application of Proposition \ref{proposition:DownwardsToL} in the given generic extension of $\VV$ now yields an elementary submodel $X$ of $\HH{\bar{\theta}}^\VV$ in $\VV$ with $\bar{\kappa}+1\subseteq X$ and the property that the embedding $\map{j\restriction X}{X}{\HH{\theta}}$ is an element of $\VV$. 
  Since $j\restriction\bar{\kappa}=\id_{\bar{\kappa}}$ and $j(\bar{\kappa})=\kappa>\bar{\kappa}$, the existence of such an embedding contradicts our assumptions on $\theta$. 
\end{proof}

\begin{corollary}\label{corollary:RemarkableShrewd}
 A weakly remarkable cardinal is remarkable if and only if it is shrewd. 
\end{corollary}

\begin{proof}
 By {\cite[Theorem 1.3]{zbMATH07149985}}, a weakly remarkable cardinal is remarkable if and only if it is $\Sigma_2$-reflecting. 
  In combination with Lemmas \ref{lemma:SRnotShrewd} and  \ref{lemma:WeaklyRemarkableShrewd}, this result directly provides the desired equivalence. 
\end{proof}

Note that {\cite[Theorem 1.4]{zbMATH07149985}} shows that every $\omega$-Erd\H{o}s cardinal is a limit of weakly remarkable cardinals that are not remarkable. By the above results,  all of these cardinals are weakly shrewd and  not shrewd.


\section{Structural reflection}

We now connect weak shrewdness with principles of structural reflection. 

 \begin{lemma}\label{lemma:ShrewdReflection}
  If $\kappa$ is weakly shrewd and $\calC$ is  a class of structures of the same type that is definable by a $\Sigma_2$-formula with parameters in $\HH{\kappa}$, then $\SRm{\kappa}{\calC}$ holds. 
 \end{lemma}
 
 \begin{proof}
 Fix a $\Sigma_2$-formula $\varphi(v_0,v_1)$ and  $z$ in $\HH{\kappa}$ with $\calC=\Set{A}{\varphi(A,z)}$. Pick a structure $B$ in $\calC$ of cardinality $\kappa$.  
  Then there exists a cardinal $\theta>\kappa$ with the property that $B\in\HH{\theta}$ and $\varphi(B,z)$ holds in $\HH{\theta}$. 
  Using Lemma \ref{lemma:WShrewdCharEmb}, we find cardinals  $\bar{\kappa}<\bar{\theta}$  and an elementary embedding $\map{j}{X}{\HH{\theta}}$ satisfying $\bar{\kappa}+1\subseteq X\prec \HH{\bar{\theta}}$, $j\restriction\bar{\kappa}=\id_{\bar{\kappa}}$, $j(\bar{\kappa})=\kappa>\bar{\kappa}$ and $B,z\in\ran{j}$.  
  Then we have $j\restriction(\HH{\bar{\kappa}}\cap X)=\id_{\HH{\bar{\kappa}}\cap X}$, and hence we know that $z\in\HH{\bar{\kappa}}$ and $j(z)=z$. 
  Pick $A\in X$ with $j(A)=B$. Then elementarity and $\Sigma_1$-absoluteness implies that $\varphi(A,z)$ holds and hence $A$ is a structure in $\calC$. 
  Since the structure $B$ has cardinality $\kappa$ in $\HH{\theta}$, we know that the structure $A$ has cardinality $\bar{\kappa}$, and the fact that $\bar{\kappa}$ is a subset of $X$ allows us to conclude that $j$ induces an elementary embedding of $A$ into $B$. 
  \end{proof}

 By combining the above lemma with the results of the previous sections, we can now prove the desired characterization of weak shrewdness through the principle $\mathrm{SR}^-$ for $\Sigma_2$-definable classes of structures.

 \begin{proof}[Proof of Theorem \ref{theorem:SRandSRCARD}]
  Fix an infinite cardinal $\kappa$.

  First, assume that $\kappa$ is the least cardinal with the property that $\SRm{\calW}{\kappa}$ holds.  
  %
 %
 Fix a cardinal $\vartheta>
 \kappa$
  with the property that  $\HH{\vartheta}$ is sufficiently elementary in $\VV$ and an elementary submodel $Y$ of $\HH{\vartheta}$ of cardinality $\kappa$ with $\kappa+1\subseteq Y$.
  Then $\vartheta$ witnesses that the structure $\langle Y,\in,\kappa\rangle$ is an element of the class $\calW$. 
  By our assumption, we can find an elementary embedding $j$ of a structure $\langle X,\in,\bar{\kappa}\rangle$ of cardinality smaller than $\kappa$ in $\calW$ into $\langle Y,\in,\kappa\rangle$. 
  Fix a cardinal $\bar{\vartheta}$ witnessing that $\langle X,\in,\bar{\kappa}\rangle$ is an element of $\calW$. 
    Then we know that  $\bar{\kappa}<\bar{\vartheta}$ are cardinals, $X$ is an elementary submodel of $\HH{\bar{\vartheta}}$  and $\map{j}{X}{\HH{\vartheta}}$ is an elementary embedding  with $\bar{\kappa}+1\subseteq X$ and $j(\bar{\kappa})=\kappa>\bar{\kappa}$.

  \begin{claim*}
   $j\restriction\bar{\kappa}=\id_{\bar{\kappa}}$. 
  \end{claim*}
  
  \begin{proof}[Proof of the Claim]
   Let $\mu\leq\bar{\kappa}$ be the minimal ordinal with $j(\mu)>\mu$. 
   Assume, towards a contradiction, that $\mu<\bar{\kappa}$. 
   Since $\mu+1\subseteq\bar{\kappa}+1\subseteq X$, elementarity implies that $\mu$ is a cardinal. 
  Moreover, since $\HH{\vartheta}$ was chosen to be sufficiently elementary in $\VV$, the minimality of $\kappa$ and the fact that $j(\mu)<\kappa$ allow us to use elementarity to find a cardinal $\mu<\theta<\bar{\vartheta}$ in $X$ and an elementary submodel $Z$ of $\HH{\theta}$ in $X$ with the property that the cardinal $j(\theta)$ witnesses that the structure $\langle j(Z),\in,j(\mu)\rangle$ is an element of $\calW$ and there is no elementary embedding from an element of $\calW$ of cardinality less than $j(\mu)$ into $\langle j(Z),\in,j(\mu)\rangle$. 
  In this situation, we know that the structure $\langle Z,\in,\mu\rangle$ is an element of $\calW$ of cardinality $\mu$ and, since $\mu\subseteq X$ implies that $Z\subseteq X$, the map $j$ induces an elementary embedding from $\langle Z,\in,\theta\rangle$ into $\langle j(Z),\in,j(\mu)\rangle$, a contradiction. 
  \end{proof}
  
  By Lemma \ref{lemma:WShrewdCharEmb}, the above claim shows that $\kappa$ is weakly shrewd in this case. 
  By Lemma \ref{lemma:ShrewdReflection}, the minimality of $\kappa$ implies that there are no weakly shrewd cardinals smaller than $\kappa$. 
  In particular, we know that (ii) implies (i).

  Now, assume that $\kappa$ is the least cardinal with the property that $\SRm{\calC}{\kappa}$ holds for every class $\calC$ of structures of the same type that is definable by a $\Sigma_2$-formula with parameters in $\HH{\kappa}$. 
   Then $\SRm{\calW}{\kappa}$ holds and we can define $\mu\leq\kappa$ to be the least cardinal such that $\SRm{\calW}{\mu}$ holds. 
   By the above computations, we know that $\mu$ is a weakly shrewd cardinal and therefore Lemma \ref{lemma:ShrewdReflection} implies that $\SRm{\calC}{\mu}$ holds for every class $\calC$ of structures of the same type that is definable by a $\Sigma_2$-formula with parameters in $\HH{\mu}$. 
 The minimality of $\kappa$ then implies that $\kappa=\mu$ and hence $\kappa$ is a weakly shrewd cardinal. 
 Moreover, another application of Lemma \ref{lemma:ShrewdReflection} shows that the minimality of $\kappa$ implies that there are no weakly shrewd cardinals below $\kappa$. 
 These arguments show that (iii) also implies (i). 
 
 Now, assume that $\kappa$ is the least weakly shrewd cardinal. 
  Then Lemma \ref{lemma:ShrewdReflection} and the above computations show that $\SRm{\calC}{\kappa}$ holds for every class $\calC$ of structures of the same type that is definable by a $\Sigma_2$-formula with parameters in $\HH{\kappa}$, and that $\SRm{\calW}{\mu}$ fails for all cardinals $\mu$ smaller than $\kappa$. 
  In combination, this shows that $\kappa$ is both the least cardinal with the property that $\SRm{\calW}{\kappa}$ holds and the least cardinal with the property that  $\SRm{\calC}{\kappa}$ holds for every class $\calC$ of structures of the same type that is definable by a $\Sigma_2$-formula with parameters in $\HH{\kappa}$. 
  This shows that (i) implies that both (ii) and (iii) hold.  
 \end{proof}

 Next, we determine the consistency strength of structural reflection for $\Sigma_2$-definable classes of structures.

\begin{proof}[Proof of Theorem \ref{theorem:ConsStrengthRefl}]
  By Corollary \ref{corollary:ShrewdWeaklyShrewd} and Lemma \ref{lemma:ShrewdReflection}, it suffices to show that, over the theory $\ZFC$, the consistency of statement (ii) implies the consistency of the statement (i). 
 Hence, assume that there exists a cardinal $\kappa$ with the property that $\SRm{\calC}{\kappa}$ holds for every class $\calC$ of structures of the same type that is definable by a $\Sigma_1(Cd)$-formula without parameters. 
 Moreover, we may assume that $0^\#$ does not exist, because otherwise a combination of  {\cite[Lemma 1.3]{schindler_2001}} with Lemma \ref{lemma:WeaklyRemarkableShrewd} and Corollary \ref{corollary:RemarkableShrewd} ensures the existence of many shrewd cardinals in $\LL$. 
   In the following, we let $\calL$ denote the first-order language extending of $\calL_\in$ by  two constant symbols  and let $\calK$ denote the class of all $\calL$-structures $\langle X,\in,\delta,\theta\rangle$ such that $\delta$ is an infinite cardinal, $\theta$ is an ordinal greater than $\delta$ and there exists a cardinal $\vartheta>\theta$ with the property that $X$ is an elementary submodel of $\LL_\vartheta$ 
 with $\delta\cup\{\delta,\theta\}\subseteq X$.   
  It is easy to see that the class $\calK$ is definable by $\Sigma_1(Cd)$-formula without parameters. 
  
  Now, pick an $\LL$-cardinal $\theta>\kappa$, a cardinal $\vartheta>\theta$ and an elementary submodel $Y$ of $\LL_\vartheta$ of cardinality $\kappa$ with $\kappa\cup\{\kappa,\theta\}\subseteq Y$. 
  Then  the $\calL$-structure $\langle Y,\in,\kappa,\theta\rangle$ is an element of $\calK$ and therefore our assumptions yield an elementary embedding $j$ of a structure $\langle X,\in,\bar{\kappa},\bar{\theta}\rangle$ of cardinality less than $\kappa$ in $\calK$ into $\langle Y,\in,\kappa,\theta\rangle$. 
  Pick a cardinal $\bar{\vartheta}$ witnessing that $\langle X,\in,\bar{\kappa},\bar{\theta}\rangle$ is contained in $\calK$. Then we know that $\bar{\kappa}$ is a cardinal smaller than $\kappa$ with $j(\bar{\kappa})=\kappa$, $\bar{\theta}$ is an $\LL$-cardinal with $j(\bar{\theta})=\theta$ and $X$ is an elementary submodel of $\LL_{\bar{\vartheta}}$ with $\bar{\kappa}\cup\{\bar{\kappa},\bar{\theta}\}\subseteq X$.

  \begin{claim*}
   $j\restriction\bar{\kappa}=\id_{\bar{\kappa}}$. 
  \end{claim*}
  
  \begin{proof}[Proof of the Claim]
   Let $\mu\leq\bar{\kappa}$ be the minimal ordinal in X with $j(\mu)>\mu$. 
   Assume, towards a contradiction that $\mu<\bar{\kappa}$ holds.  
   Note that, since $j(\bar{\kappa})=\kappa$ holds and $\bar{\kappa}+1\subseteq X$ implies $\LL_{\bar{\kappa}}\cup\{\LL_{\bar{\kappa}}\}\subseteq X$, we know that the map $\map{j\restriction\LL_{\bar{\kappa}}}{\LL_{\bar{\kappa}}}{\LL_\kappa}$ is an elementary embedding between transitive structures. 
   In this situation, our assumption implies that this embedding has a critical point. 
   But then the fact that $\bar{\kappa}$ is a cardinal implies that $\betrag{\crit{j\restriction\LL_{\bar{\kappa}}}}\leq \mu<\bar{\kappa}$ and hence {\cite[Theorem 18.27]{MR1940513}} shows that $0^\#$ exists, a contradiction. 
  \end{proof}

   Since the above claim  shows that $\map{j}{Y}{\LL_\vartheta}$ is an elementary embedding with $\bar{\kappa}\cup\{\bar{\kappa},\bar{\theta}\}\subseteq Y\prec\LL_{\bar{\vartheta}}$, $j\restriction\bar{\kappa}=\id_{\bar{\kappa}}$ and $j(\bar{\kappa})=\kappa$, we can apply Proposition \ref{proposition:DownwardsToL} to find an elementary submodel $X$ of $\LL_\theta$ in $\LL$ with $X\subseteq Y$ and the property that the map $j\restriction X$ is also an element of $\LL$.  
  These computations now allow us to conclude that $\kappa$ is a weakly shrewd cardinal in $\LL$. By Corollary \ref{corollary:EquiInacc}, this argument shows that, over  $\ZFC$, the consistency of the second statement listed in the theorem implies the consistency of the first statement listed there.  
\end{proof}

 Using ideas from the above proofs, we can show that a failure of the converse implication of Lemma \ref{lemma:ShrewdReflection} has non-trivial consistency strength. 
 This argument again makes use of the class $\calW$ defined in Section \ref{section:Intro}.

 \begin{lemma}
  If $\SRm{\calW}{\kappa}$ holds for some infinite cardinal $\kappa$ that is  not a weakly shrewd cardinal, then $0^\#$ exists. 
 \end{lemma}
 
 \begin{proof}
  By Lemma \ref{lemma:WShrewdCharEmb}, our assumptions yield a cardinal $\theta>\kappa$ with the property that for all  cardinals $\bar{\kappa}<\bar{\theta}$, there is no elementary embedding $\map{j}{X}{\HH{\theta}}$ with $\bar{\kappa}+1\subseteq X\prec\HH{\bar{\theta}}$, $j\restriction\bar{\kappa}=\id_{\bar{\kappa}}$ and $j(\bar{\kappa})=\kappa>\bar{\kappa}$. 
  As in the proof of Theorem \ref{theorem:SRandSRCARD}, we can use these assumption to find cardinals $\bar{\kappa}<\bar{\theta}$ and an elementary embedding $\map{j}{X}{\HH{\theta}}$ with $\bar{\kappa}+1\subseteq X\prec\HH{\bar{\theta}}$ and $j(\bar{\kappa})=\kappa$. 
  By our assumptions, we know that $j\restriction\bar{\kappa}\neq\id_{\bar{\kappa}}$. 
 Since $\LL_{\bar{\kappa}}\cup\{\LL_{\bar{\kappa}}\}\subseteq X$, we now know that $\map{j\restriction\LL_{\bar{\kappa}}}{\LL_{\bar{\kappa}}}{\LL_\kappa}$ is an elementary embedding between transitive structures that has a critical point. 
 As in the proof of Theorem \ref{theorem:ConsStrengthRefl}, the fact that $\bar{\kappa}$ is a cardinal and  $\betrag{\crit{j\restriction\LL_{\bar{\kappa}}}}<\bar{\kappa}$ now allows us to apply {\cite[Theorem 18.27]{MR1940513}} to conclude that $0^\#$ exists.  
\end{proof}

 By combining the above observation with Lemma \ref{lemma:ShrewdReflection}, we can now conclude that, in the constructible universe $\LL$, weak shrewdness is equivalent to the validity of the principle $\mathrm{SR}^-$ for $\Sigma_2$-definable classes.

 \begin{corollary}\label{corollary:0-ShrewdRefl}
  If $\VV=\LL$ holds, then the following statements are equivalent for every infinite cardinal $\kappa$: 
  \begin{enumerate}
   \item $\kappa$ is a weakly shrewd cardinal. 
   
   \item $\SRm{\calW}{\kappa}$ holds. 
   
   \item $\SRm{\calC}{\kappa}$ holds for every class $\calC$ of structures of the same type that is definable by a $\Sigma_2$-formula with parameters in $\HH{\kappa}$.  \qed
  \end{enumerate}
 \end{corollary}

 In addition, we can also use the above lemma to motivate the statement of Theorem \ref{theorem:HigherReflection} by showing that the existence of a weakly shrewd cardinal does not imply the existence of a reflection point for classes of structures defined by $\Sigma_3$-formulas. 
 Note that the above results show that, over $\ZFC$, the consistency of the existence of a shrewd cardinal implies the consistency of the assumptions of the following corollary.

\begin{corollary}\label{corollary:ShrewdNoMoreRefl}
 If $\VV=\LL$ holds and there exists a single weakly shrewd cardinal, then there is no cardinal $\rho$ with the property that $\SRm{\calC}{\rho}$ holds for every class $\calC$ of structures of the same type that is definable by a $\Sigma_3$-formula without parameters. 
\end{corollary}
 
\begin{proof}
 Let $\kappa$ denote the unique weakly shrewd cardinal.
 Then it is easy to see that the set $\{\kappa\}$ is definable by a $\Sigma_3$-formula without parameters. 
 Therefore, there is a non-empty class $\calC$ of structures of the same type that is definable by a $\Sigma_3$-formula without parameters and consists of structures of cardinality $\kappa$. 
 In particular, we know that $\SRm{\calC}{\kappa}$ fails. 
 
 Now, assume, towards a contradiction, that there is a cardinal $\rho$ with the property that  $\SRm{\calC}{\rho}$ holds for every class $\calC$ of structures of the same type that is definable by a $\Sigma_3$-formula without parameters. 
 An application of Corollary \ref{corollary:0-ShrewdRefl} then directly shows that $\rho$ is a weakly shrewd cardinal and hence we can conclude that $\kappa=\rho$ holds. 
 Since $\SRm{\calC}{\kappa}$ fails, this yields a contradiction. 
\end{proof}

 In the remainder of this section, we show that hyper-shrewd cardinals imply the existence of cardinals with strong structural reflection properties.

 \begin{proof}[Proof of Theorem \ref{theorem:HigherReflection}]
  Let $\kappa$ be a weakly shrewd cardinal that is not shrewd.  
  Then Lemma \ref{lemma:SRnotShrewd} already shows that there is a cardinal $\delta>\kappa$ with the property that the set $\{\delta\}$ is definable by a $\Sigma_2$-formula with parameters in $\HH{\kappa}$. 
  For the remainder of this proof, fix such a cardinal $\delta$. 
  
  (ii) Given a natural number $n>0$ and an ordinal $\alpha<\kappa$,  assume, towards a contradiction, that for every cardinal $\alpha<\rho<\delta$, there exists a class $\calC$ of structures of the same type that is definable by a $\Sigma_n$-formula with parameters in $\HH{\rho}$ such that $\SRm{\calC}{\rho}$ fails.\footnote{Note that, by using an universal $\Sigma_n$-formula, this statement can be expressed by a single $\calL_\in$-formula with parameters $\alpha$ and $\delta$.}
  Now, let $\varepsilon$ denote the least strong limit cardinal above $\delta$ with the property that, in $\HH{\varepsilon}$, for every cardinal $\alpha<\rho<\delta$, there exists a class $\calC$ of structures of the same type that is definable by a $\Sigma_n$-formula with parameters in $\HH{\rho}$ such that $\SRm{\calC}{\rho}$ fails. 
   Then the definability of $\delta$ ensures that the set $\{\varepsilon\}$ is also definable by a $\Sigma_2$-formula with parameters in $\HH{\kappa}$. 
  In this situation, in $\HH{\varepsilon}$, there is  a class $\calC$ of structures of the same type and a structure $B$ of cardinality $\kappa$ in $\calC$ such that $\calC$ is definable by a $\Sigma_n$-formula $\varphi(v_0,v_1)$ with parameter $z\in\HH{\kappa}$ and there is no elementary embedding of a structure of cardinality less than $\kappa$ in $\calC$ into $B$. 
 Since Lemma \ref{lemma:SRnonShrewdFix} shows that $\kappa$ is $\varepsilon$-hyper-shrewd, there are  cardinals $\bar{\kappa}<\kappa<\varepsilon<\bar{\theta}$ and an elementary embedding $\map{j}{X}{\HH{\theta}}$ with $\bar{\kappa}\cup\{\bar{\kappa},\varepsilon\}\subseteq X\prec\HH{\bar{\theta}}$, $j\restriction\bar{\kappa}=\id_{\bar{\kappa}}$,  $j(\bar{\kappa})=\kappa$, $j(\varepsilon)=\varepsilon$ and $B,z\in\ran{j}$.
 Moreover, since $j\restriction(\HH{\bar{\kappa}}\cap X)=\id_{\HH{\bar{\kappa}}\cap X}$, we know that $z\in \HH{\bar{\kappa}}\cap X$ and $j(z)=z$. 
 Pick $A\in\HH{\bar{\theta}}\cap X$ with $j(A)=B$. 
  Then $A\in\HH{\varepsilon}$ and, since both $\varepsilon$ and $z$ are fixed points of $j$, we know that $\varphi(A,z)$ holds in $\HH{\varepsilon}$. 
  In particular, we can conclude that $A$ is a structure of cardinality $\bar{\kappa}$ in $\calC$ and, since $\bar{\kappa}$ is a subset of $X$,  the embedding  $j$ induces an elementary embedding of $A$ into $B$. 
  But now, the fact $A$ and $B$ are both contained in $\HH{\varepsilon}$ implies that this embedding is also an element of $\HH{\varepsilon}$, a contradiction.  
  
  (iii) Assume that $\kappa$ is inaccessible and $0^\#$ does not exist. 
  Using Lemma \ref{lemma:SRnonShrewdFix}, we find a cardinal $\theta>2^\delta$, cardinals $\bar{\kappa}<\kappa<\delta<\bar{\theta}$ and an elementary embedding $\map{j}{X}{\HH{\theta}}$ with $\bar{\kappa}\cup\{\bar{\kappa},\delta\}\subseteq X\prec\HH{\bar{\theta}}$, $j\restriction\bar{\kappa}=\id_{\bar{\kappa}}$,  $j(\bar{\kappa})=\kappa$ and  $j(\delta)=\delta$.
   Let $\varepsilon$ denote the minimal element in $X\cap[\bar{\kappa},\delta]$ satisfying $j(\varepsilon)=\varepsilon$. 
   
  \begin{claim*}
   $\varepsilon$ is an inaccessible cardinal.  
  \end{claim*}
   
   \begin{proof}[Proof of the Claim]
    First, assume, towards a contradiction, that $\varepsilon$ is singular. 
    Since we have $\betrag{\varepsilon}=\betrag{\varepsilon}^{\HH{\theta}}=\betrag{\varepsilon}^{\HH{\bar{\theta}}}\in X$, it follows that $j(\betrag{\varepsilon})=\betrag{\varepsilon}\geq\bar{\kappa}$ and hence the minimality of $\varepsilon$ implies that $\varepsilon$ is a cardinal. 
     Then the non-existence of $0^\#$ implies that $\varepsilon$ is a singular cardinal in $\LL$. 
     Set $\lambda=\cof{\varepsilon}^\LL<\varepsilon\leq\delta$ and let $\map{c}{\lambda}{\varepsilon}$ be the $<_\LL$-least cofinal map from $\lambda$ to $\varepsilon$ in $\LL$. 
  Since $\LL_\delta\subseteq\HH{\theta}\cap\HH{\bar{\theta}}$, we now know that $\lambda,c\in X$, $j(\lambda)=\lambda$ and $j(c)=c$. But this implies that $\lambda<\bar{\kappa}$ and hence $\lambda\subseteq X$. 
  Pick $\xi<\lambda$ with $c(\xi)>\bar{\kappa}$. Then $$\bar{\kappa} ~ < ~ \kappa ~ = ~ j(\bar{\kappa}) ~ < ~ j(c(\xi)) ~ = ~ j(c)(j(\xi)) ~ = ~ c(\xi) ~ < ~ \varepsilon,$$ contradicting the minimality of $\varepsilon$. 
  This shows that $\varepsilon$ is a regular cardinal. 
  
  Now, assume, towards a contradiction, that $\varepsilon$ is not a strong limit cardinal and let $\mu<\varepsilon$ be the least cardinal satisfying $2^\mu\geq\varepsilon$. 
Then the inaccessibility of $\kappa$ implies that $\mu\geq\kappa$. 
 Since elementarity implies that $2^\delta\in\HH{\theta}\cap\HH{\bar{\theta}}$, we know that $\mu$ is the least cardinal with $2^\mu\geq\varepsilon$ in both $\HH{\theta}$ and $\HH{\bar{\theta}}$. 
 But this implies that $\mu\in X\cap[\kappa,\varepsilon)$ with $j(\mu)=\mu$, a contradiction.     
   \end{proof}
   
  Now, assume, towards a contradiction, that there exists $\calC\subseteq\VV_\varepsilon$, $a\in\mathsf{Fml}$ and $z\in\HH{\kappa}$ such that $\calC=\Set{A\in\VV_\varepsilon}{\mathsf{Sat}(\VV_\varepsilon,\langle A,z\rangle,a)}$ and, in $\VV_\varepsilon$,\footnote{This statement about the structure $\langle\VV_\varepsilon,\in\rangle$ is again stated using the formalized satisfaction relation $\mathsf{Sat}$.} the class $\calC$ consists of structures of the same type and $\SRm{\calC}{\kappa}$ fails. 
   By elementarity, we find $\calC\in \VV_{\varepsilon+1}\cap X$,  $a\in\mathsf{Fml}\subseteq X$ and $z\in\HH{\bar{\kappa}}\cap X$ such that the elements $j(\calC)$, $j(a)$ and $j(z)$ witness the above statement.  
  Since  we have $j(a)=a$ and $j(z)=z$, we also know that $j(\calC)=\calC$. 
  Another application of elementarity now yields $B\in \calC\cap X$ such that $\mathsf{Sat}(\VV_\varepsilon,\langle j(B),z\rangle,a)$ holds and, in $\VV_\varepsilon$, the structure $j(B)$ witnesses the failure of $\SRm{\calC}{\kappa}$.  
 We then know that  $\mathsf{Sat}(\VV_\varepsilon,\langle B,z\rangle,a)$ holds and hence $B$ is also an element of $\calC$. 
 Since $B$ is a structure of cardinality $\bar{\kappa}$ in $\VV_\varepsilon$, it follows that the domain of $B$ is a subset of $X$ and, by the inaccessibility of $\varepsilon$, the restriction of $j$ to this domain is an element of $\VV_\varepsilon$. But this restricted map is an elementary embedding from $B$ into $j(B)$ in $\VV_\varepsilon$, a contradiction. 
   The above arguments show that, in $\VV_\varepsilon$, the principle $\SRm{\calC}{\kappa}$ holds for every class $\calC$ that is defined by a formula using parameters from $\HH{\kappa}$. 
\end{proof}

 We now use Theorems  \ref{theorem:WShrewNonShrewGCH} and \ref{theorem:HigherReflection} to show that the consistency strength of the existence of cardinals with maximal local structural reflection properties is strictly smaller than the existence of a weakly shrewd cardinal that is not shrewd. 
 Moreover,  by combining Theorem \ref{theorem:HigherReflection} with the compactness theorem, we can prove that the principle $\mathrm{SR}^-$ cannot be used to characterize large cardinal properties that imply strong inaccessibility.

 \begin{proof}[Proof of Theorem \ref{theorem:ConsHyper}]
  (i) Let $\kappa$ be a weakly shrewd cardinal that is not shrewd. 
  Then Lemma \ref{lemma:SRnonShrewdFix} shows that $\kappa$ is $\delta$-hyper-shrewd for some cardinal $\delta>\kappa$ and we can apply Corollary \ref{corollary:Downwards} to conclude that $\kappa$ is  $\delta$-hyper-shrewd  in $\LL$. 
 Let $G$ be $\Add{\delta^+}{1}^\LL$-generic over $\LL$.
 Then Lemma \ref{lemma:MakeDefShrewd} shows that,  in $\LL[G]$, the set $\{\delta\}$ is definable by a $\Sigma_2$-formula without parameters and $\kappa$ is an inaccessible  weakly shrewd cardinal that is not shrewd. 
  Since $0^\#$ does not exist in $\LL[G]$ and the partial order $\Add{\delta^+}{1}^\LL$ is ${<}\delta$-closed in $\LL$, we can apply the last part of Theorem \ref{theorem:HigherReflection} in $\LL[G]$ to find $\kappa<\varepsilon<\delta$ with the property that $\langle\LL_\varepsilon,\in,\kappa\rangle$ is a transitive model of the formalized $\calL_{\dot{c}}$-theory $\ZFC  +  \Set{\mathsf{SR}^-_n}{0<n<\omega}$ with respect to some canonical formalized satisfaction predicate. 
 Since such a satisfaction predicate can be defined by a $\Delta_1^{\ZFC^-}$-formula, the model $\langle\LL_\varepsilon,\in,\kappa\rangle$ also has these properties in both $\LL$ and $\VV$. 
These computations prove the first part of the theorem.

 (ii) Assume that the existence of a weakly shrewd cardinal that is not  shrewd is consistent with the axioms of $\ZFC$. 
  By Theorem \ref{theorem:WShrewNonShrewGCH}, this assumption implies that the existence of a weakly shrewd cardinal smaller than $2^{\aleph_0}$ is consistent with $\ZFC$. 
  Since the set $\{2^{\aleph_0}\}$ is always definable by a $\Sigma_2$-formula without parameters, we can now apply Theorem \ref{theorem:HigherReflection} to show that for all $0<n<\omega$, the $\calL_{\dot{c}}$-theory $\ZFC+\mathsf{SR}^-_n+\anf{\dot{\kappa}<2^{\aleph_0}}$ is consistent. 
  By the \emph{Compactness Theorem}, this allows us to conclude that our assumption implies the consistency of the $\calL_{\dot{c}}$-theory $$\ZFC ~ + ~ \Set{\mathsf{SR}^-_n}{0<n<\omega} ~ + ~ \anf{\hspace{2.3pt}\dot{\kappa}<2^{\aleph_0}}.$$ 
  
  Now, assume that the above $\calL_{\dot{c}}$-theory is consistent. 
  By Theorem \ref{theorem:SRandSRCARD}, this implies that $\ZFC$ is consistent with the existence of a weakly shrewd cardinal smaller than $2^{\aleph_0}$ and therefore $\ZFC$ is consistent with the existence of a weakly shrewd cardinal that is not shrewd. 
 \end{proof}


 \section{Fragments of weak shrewdness}

 Motivated by Rathjen's definition of \emph{$A$-$\eta$-shrewd cardinals} (see {\cite[Definition 2.2]{zbMATH02168085}}), we now study restrictions of weak shrewdness and derive embedding characterizations for the resulting large cardinal notions. 
 Together with the concept of local $\Sigma_n(R)$-classes, this analysis will allow us to characterize several classical notions from the lower part of the large cardinal hierarchy through the principle $\textrm{SR}^-$ in the next section. 
 Moreover, our results will allow us to show that Hamkins' \emph{weakly compact embedding property} is equivalent to L{\'e}vy's notion of \emph{weak $\Pi^1_1$-indescribability}.

  Remember that $\calL_{\dot{A}}$ denotes the first-order language that extends $\calL_\in$ by a unary predicate symbol $\dot{A}$.

\begin{definition}
 Given a class $R$, a natural number $n>0$ and a cardinal $\theta$, an infinite cardinal $\kappa<\theta$ is \emph{weakly $(\Sigma_n,R,\theta)$-shrewd} if for every $\Sigma_n$-formula $\Phi(v_0,v_1)$ in $\calL_{\dot{A}}$ and every  $A\subseteq\kappa$ with the property that $\Phi(A,\kappa)$ holds in $\langle\HH{\theta},\in,R\rangle$,  there exist cardinals $\bar{\kappa}<\bar{\theta}$ such that $\bar{\kappa}<\kappa$ and $\Phi(A\cap{\bar{\kappa}},\bar{\kappa})$ holds in $\langle\HH{\bar{\theta}},\in,R\rangle$. 
\end{definition}

 It is easy to see that a cardinal $\kappa$ is weakly shrewd if and only if it is weakly $(\Sigma_n,R,\theta)$-shrewd for all cardinals $\theta>\kappa$, all $0<n<\omega$ and all classes $R$ that are definable by $\Pi_1$-formulas with parameters in $\HH{\kappa}$. 
 The following variation of Lemma \ref{lemma:WShrewdCharEmb} provides a characterization of  restricted weak shrewdness in terms of elementary embeddings between set-sized structures:

\begin{lemma}\label{lemma:EmbLEvelShrewd}
 The following statements are equivalent for all classes $R$, all natural numbers $n>0$ and all infinite cardinals $\kappa<\theta$: 
 \begin{enumerate}
  \item $\kappa$ is weakly $(\Sigma_n,R,\theta)$-shrewd. 
  
  \item For every $A\subseteq\kappa$, there exist cardinals $\bar{\kappa}<\bar{\theta}$ and a $\Sigma_n$-elementary embedding $\map{j}{\langle X,\in,R\rangle}{\langle\HH{\theta},\in,R\rangle}$ satisfying $\langle X,\in,R\rangle\prec_{\Sigma_{n-1}}\langle\HH{\bar{\theta}},\in,R\rangle$,\footnote{In the setting of this lemma, the notions of $\Sigma_n$-elementary embeddings and submodels are defined through the absoluteness of $\Sigma_n$-formulas in the extended language $\calL_{\dot{A}}$.}   $\bar{\kappa}+1\subseteq X$, $j\restriction\bar{\kappa}=\id_{\bar{\kappa}}$,  $j(\bar{\kappa})=\kappa>\bar{\kappa}$ and $A\in\ran{j}$. 
  
  \item For every $z\in\HH{\theta}$, there exist cardinals $\bar{\kappa}<\bar{\theta}$    and an elementary embedding $\map{j}{\langle X,\in,R\rangle}{\langle\HH{\theta},\in,R\rangle}$ satisfying $\langle X,\in,R\rangle\prec_{\Sigma_{n-1}}\langle\HH{\bar{\theta}},\in,R\rangle$, $\bar{\kappa}+1\subseteq X$, $j\restriction\bar{\kappa}=\id_{\bar{\kappa}}$,  $j(\bar{\kappa})=\kappa>\bar{\kappa}$ and $z\in\ran{j}$.   
  \end{enumerate}
\end{lemma}

\begin{proof}
 First, assume that (ii) holds. Fix a $\Sigma_n$-formula $\Phi(v_0,v_1)$ in $\calL_{\dot{A}}$ and a subset $A$ of $\kappa$ such that $\Phi(A,\kappa)$ holds in $\langle\HH{\theta},\in,R\rangle$. 
  Pick cardinals $\bar{\kappa}<\bar{\theta}$ and a $\Sigma_n$-elementary embedding $\map{j}{\langle X,\in,R\rangle}{\langle\HH{\theta},\in,R\rangle}$ such that $\langle X,\in,R\rangle\prec_{\Sigma_{n-1}}\langle\HH{\bar{\theta}},\in,R\rangle$,   $\bar{\kappa}+1\subseteq X$, $j\restriction\bar{\kappa}=\id_{\bar{\kappa}}$,  $j(\bar{\kappa})=\kappa>\bar{\kappa}$ and $A\in\ran{j}$. 
  Then $A\cap\bar{\kappa}\in X$ and $j(A\cap\bar{\kappa})=A$. Therefore elementarity implies that the $\Sigma_n$-statement $\Phi(A\cap\bar{\kappa},\bar{\kappa})$ holds in $\langle X,\in,R\rangle$ and, since our assumptions ensure that $\Sigma_n$-statements are upwards-absolute from $\langle X,\in,R\rangle$ to $\langle\HH{\bar{\theta}},\in,R\rangle$, we can conclude that this statement also holds in $\langle\HH{\bar{\theta}},\in,R\rangle$. 
  These computations show that (i) holds in this case. 
  
  Now, assume that (i) holds and fix an element $z$ of $\HH{\theta}$. 
  Pick an elementary submodel $\langle Y,\in,R\rangle$ of $\langle\HH{\theta},\in,R\rangle$ of cardinality $\kappa$ with $\kappa\cup\{\kappa,z\}\subseteq Y$, and a bijection $\map{b}{\kappa}{Y}$ with $b(0)=\kappa$, $b(1)=\langle z,\kappa\rangle$ and $b(\omega\cdot(1+\alpha))=\alpha$ for all $\alpha<\kappa$. 
  Let  $\mathsf{Fml}_*$ denote the class of all formalized $\calL_{\dot{A}}$-formulas and let $\mathsf{Sat}_*$ denote the formalized satisfaction relation for $\calL_{\dot{A}}$-structures. 
  Then the classes $\mathsf{Fml}_*$ and $\mathsf{Sat}_*$ are both defined by  $\Sigma_1$-formulas in $\calL_\in$.   
   Fix a recursive enumeration $\seq{a_l}{l<\omega}$ of the class $\mathsf{Fml}_*$ and let $A$ be the subset of $\kappa$ consisting of all ordinals of the form $\goedel{l}{\alpha_0,\ldots,\alpha_{m-1}}$ with the property that $l<\omega$, $\alpha_0,\ldots,\alpha_{m-1}<\kappa$, $a_l$ codes a formula with $m$ free variables and $\mathsf{Sat}_*(Y,R\cap Y,\langle b(\alpha_0),\ldots,b(\alpha_{m-1})\rangle,a_l)$ holds. 
   Let $\psi(v_0,v_1)$ be a universal $\Pi_n$-formula  in $\calL_{\dot{A}}$ (as constructed in {\cite[Section 1]{JensenFine}}).\footnote{Note that for every infinite cardinal $\vartheta$ and every class $R$, the structure $\langle\HH{\vartheta},\in,R\rangle$ is closed under functions which are \emph{rudimentary in $R$} . This shows that for every $\Pi_n$-formula $\varphi(v)$ in $\calL_{\dot{A}}$, we can find  $k<\omega$ such that $$\langle X,\in,R\rangle\models\forall x ~ [\varphi(x)\longleftrightarrow\psi(k,x)]$$ holds for every class $R$, every infinite cardinal $\vartheta$ and every elementary submodel $\langle X,\in,R\rangle$ of $\langle\HH{\vartheta},\in,R\rangle$.} 
   Now, pick a  $\Sigma_n$-formula\footnote{Note that, given $0<n<\omega$, a standard induction shows that the class of all $\calL_{\dot{A}}$-formulas $\varphi(v_0,\ldots,v_{m-1})$ with the property that there exists a $\Sigma_n$-formula $\psi(v_0,\ldots,v_{m-1})$ in $\calL_{\dot{A}}$ such that $\ZFC$ proves that $$\langle\HH{\vartheta},\in,R\rangle\models\forall x_0,\ldots,x_{m-1} ~ [\varphi(x_0,\ldots,x_{m-1})\longleftrightarrow\psi(x_0,\ldots,x_{m-1})]$$ holds for every class $R$ and every infinite cardinal $\vartheta$ is closed under conjunctions, disjunctions, bounded universal quantification and unbounded existential quantification.} $\Phi(v_0,v_1)$ in $\calL_{\dot{A}}$ such  that $\Phi(B,\delta)$ holds in structures of the form $\langle\HH{\vartheta},\in,R\rangle$ if and only if $\delta<\vartheta$ is a limit ordinal and, in $\HH{\vartheta}$,  there exists a set $X$ and a bijection $\map{b}{\delta}{X}$ such that the following statements hold: 
    \begin{enumerate}    
     \item If $x\in X$ and $k<\omega$ such that $\psi(k,x)$ holds in $\langle\HH{\vartheta},\in,R\rangle$, then $\psi(k,x)$ holds in $\langle X,\in,R\rangle$. 
     
     \item $\delta+1\subseteq X$, $b(0)=\delta$ and $b(\omega\cdot(1+\alpha))=\alpha$ for all $\alpha<\delta$. 
     
     \item Given      $\alpha_0,\ldots,\alpha_{m-1}<\delta$ and $l<\omega$ such that $a_l$ codes a formula with $m$ free variables,  we have $$\goedel{l}{\alpha_0,\ldots,\alpha_{m-1}}\in B ~  \Longleftrightarrow ~ \mathsf{Sat}_*(X,R\cap X,\langle b(\alpha_0),\ldots,b(\alpha_{m-1})\rangle,a_l).$$
  \end{enumerate}
  
  Then $\Phi(A,\kappa)$ holds in $\langle\HH{\theta},\in,R\rangle$ and hence we can find cardinals $\bar{\kappa}<\bar{\theta}$ such that  $\bar{\kappa}<\kappa$ and $\Phi(A\cap{\bar{\kappa}},\bar{\kappa})$ holds in $\langle\HH{\bar{\theta}},\in,R\rangle$. 
 Pick a set $X\in\HH{\bar{\theta}}$ and a bijection $\map{\bar{b}}{\bar{\kappa}}{X}$ witnessing this. 
 Then the definition of the formula $\Phi$ ensures that all $\Pi_n$-formulas in $\calL_{\dot{A}}$ are downwards-absolute from $\langle\HH{\bar{\theta}},\in,R\rangle$ to $\langle X,\in,R\rangle$ and this directly allows us to conclude that $\langle X,\in,R\rangle\prec_{\Sigma_{n-1}}\langle\HH{\bar{\theta}},\in,R\rangle$. 
 Define $$\map{j=b\circ\bar{b}}{X}{\HH{\theta}}.$$
 Then we know that $j$ is an elementary embedding of $\langle X,\in,R\rangle$ into $\langle\HH{\theta},\in,R\rangle$ satisfying  $j\restriction\bar{\kappa}=\id_{\bar{\kappa}}$, $j(\bar{\kappa})=\kappa$ and $z\in\ran{j}$. 
   This shows that (iii) holds in this case. 
\end{proof}

We now derive some easy consequences of the equivalences established above.

\begin{corollary}\label{corollary:FragShrewdWeaklyInacc}
 Given infinite cardinals $\kappa<\theta$, if $\kappa$ is weakly $(\Sigma_1,\emptyset,\theta)$-shrewd, then $\kappa$ is weakly inaccessible.  
\end{corollary}

\begin{proof}
 Assume, towards a contradiction, that $\kappa$ is not weakly inaccessible. 
 Since Lemma \ref{lemma:EmbLEvelShrewd} directly implies that all weakly $(\Sigma_1,\emptyset,\theta)$-shrewd cardinals are uncountable limit cardinals, this assumption allows us to find  a cofinal subset $A$ of $\kappa$ of order-type $\lambda<\kappa$. 
  By Lemma \ref{lemma:EmbLEvelShrewd}, there exist cardinals  $\bar{\kappa}<\bar{\theta}$    and an elementary embedding $\map{j}{X}{\HH{\theta}}$ satisfying $X\prec_{\Sigma_0}\HH{\bar{\theta}}$, $\bar{\kappa}+1\subseteq X$, $j\restriction\bar{\kappa}=\id_{\bar{\kappa}}$,  $j(\bar{\kappa})=\kappa>\bar{\kappa}$ and $A,\lambda\in\ran{j}$.   
  But then we have $\lambda\cup\{\lambda,A\cap\bar{\kappa}\}\subseteq X$, $j\restriction(\lambda+1)=\id_{\lambda+1}$ and $j(A\cap\bar{\kappa})=A$. 
  By elementarity, this implies that $\otp{A}=\lambda=\otp{A\cap\bar{\kappa}}$ and hence $A\subseteq\bar{\kappa}$, a contradiction. 
\end{proof}

 Next, we observe that, in the case $\theta=\kappa^+$, the Lemma \ref{lemma:EmbLEvelShrewd} yields non-trivial elementary embeddings between transitive structures. 
 Note that the same argument as for $\Sigma_0$-formulas in $\calL_\in$ shows that $\langle M,\in,R\rangle\prec_{\Sigma_0}\langle N,\in,R\rangle$ holds for every class $R$ and all transitive classes $M$ and $N$ with $M\subseteq N$.

 \begin{corollary}\label{corollary:KappaPlusEmbe}
 Let $R$ be a class, let $n>0$ be a natural number, let  $\kappa$ be a weakly $(\Sigma_n,R,\kappa^+)$-shrewd cardinal and let $z\in\HH{\kappa^+}$. 
  Then there exists a transitive set $N$ and a non-trivial elementary embedding $\map{j}{\langle N,\in,R\rangle}{\langle\HH{\kappa^+},\in,R\rangle}$ with the property that $\crit{j}$ is a cardinal, $\langle N,\in,R\rangle\prec_{\Sigma_{n-1}}\langle\HH{\crit{j}^+},\in,R\rangle$, $j(\crit{j})=\kappa$ and $z\in\ran{j}$. 
 \end{corollary}
 
 \begin{proof} 
  With the help of Lemma \ref{lemma:EmbLEvelShrewd}, we can  find cardinals $\bar{\kappa}<\bar{\theta}$    and an elementary embedding $\map{j}{\langle N,\in,R\rangle}{\langle\HH{\kappa^+},\in,R\rangle}$ with $\langle N,\in,R\rangle\prec_{\Sigma_{n-1}}\langle\HH{\bar{\theta}},\in,R\rangle$, $\bar{\kappa}+1\subseteq N$, $j\restriction\bar{\kappa}=\id_{\bar{\kappa}}$,  $j(\bar{\kappa})=\kappa>\bar{\kappa}$ and $z\in\ran{j}$. 
  In this situation,   elementarity implies that $N=\HH{\bar{\kappa}^+}^N=\HH{\bar{\kappa}^+}\cap N\subseteq\HH{\bar{\kappa}^+}$ and, since $\bar{\kappa}\subseteq N$, this shows that $N$ is transitive. 
  In particular, we know that $\langle N,\in,R\rangle\prec_{\Sigma_0}\langle\HH{\bar{\kappa}^+},\in,R\rangle$ and this completes the proof in the case $n=1$.  
  Next,  if $n>2$, then elementarity directly implies that $\bar{\theta}=\bar{\kappa}^+$ and this shows that $N$ also possesses the desired properties in this case.  
  Finally, if $n=2$, then we have $\langle N,\in,R\rangle\prec_{\Sigma_1}\langle\HH{\bar{\theta}},\in,R\rangle$,  transitivity implies that all $\Sigma_1$-formulas are upwards-absolute from $\langle N,\in,R\rangle$ to $\langle\HH{\bar{\kappa}^+},\in,R\rangle$ and from  $\langle\HH{\bar{\kappa}^+},\in,R\rangle$ to $\langle\HH{\bar{\theta}},\in,R\rangle$, and therefore we can conclude that $\langle N,\in,R\rangle\prec_{\Sigma_{1}}\langle\HH{\bar{\kappa}^+},\in,R\rangle$ holds in this case. 
\end{proof}

 The next lemma will allow us to show that both weak inaccessibility and weak Mahloness are equivalent to certain restrictions of weak shrewdness.

\begin{lemma}\label{lemma:ReflCardReg}
 Given a class $R$ of cardinals, the following statements are equivalent for every cardinal $\kappa$ in $R$: 
 \begin{enumerate}
  \item $\kappa$ is regular and  the set $R\cap\kappa$ is stationary in $\kappa$. 
  
  \item $\kappa$ is weakly $(\Sigma_1,R,\kappa^+)$-shrewd. 
 \end{enumerate} 
\end{lemma}

\begin{proof}
 First, assume that (i) holds and fix $z\in\HH{\kappa^+}$. 
 With the help of an elementary chain of submodels, we can use the stationarity of $R\cap\kappa$ in $\kappa$ to find a cardinal $\bar{\kappa}\in R\cap\kappa$ and an elementary substructure $\langle X,\in,R\rangle$ of   $\langle\HH{\kappa^+},\in,R\rangle$ of cardinality $\bar{\kappa}$ such that $X\cap\kappa=\bar{\kappa}$ and $\tc{\{z\}}\subseteq X$. 
   Let $\map{\pi}{X}{N}$ denote the corresponding transitive collapse. 
 
 \begin{claim*}
  $\pi[R\cap X]=N\cap R$. 
 \end{claim*}
 
 \begin{proof}[Proof of the Claim]
  In one direction, if $\delta\in R\cap X\subseteq\kappa^+$, then the fact that $R$ consists of cardinals implies that either $\delta=\kappa$ and $\pi(\delta)=\bar{\kappa}\in R$, or $\delta\in X\cap\kappa=\bar{\kappa}$ and $\pi(\delta)=\delta\in R$. 
  In the other direction, if $\alpha\in N$ with $\pi^{{-}1}(\alpha)\notin R$, then either $\alpha<\bar{\kappa}$ and $\alpha=\pi^{{-}1}(\alpha)\notin R$, or $\alpha>\bar{\kappa}$, $\alpha$ has cardinality $\bar{\kappa}$ in $N$ and we can again make use of the fact that $R$ consists of cardinals to conclude that $\alpha\notin R$. 
 \end{proof} 

 The above claim now shows that $\map{\pi}{\langle N,\in,R\rangle}{\langle X,\in,R\rangle}$ is an isomorphism and this directly implies that $\map{\pi^{{-}1}}{\langle  N,\in,R\rangle}{\langle\HH{\kappa^+},\in,R\rangle}$ is an elementary embedding with $\crit{\pi^{{-}1}}=\bar{\kappa}$, $\pi^{{-}1}(\bar{\kappa})=\kappa$ and $z\in\ran{\pi^{{-}1}}$. 
 Moreover, we know that $N$ is a transitive set of cardinality $\bar{\kappa}$ and hence $\langle N,\in,R\rangle\prec_{\Sigma_0}\langle\HH{\bar{\kappa}^+},\in,R\rangle$. 
 By Lemma \ref{lemma:EmbLEvelShrewd}, this shows that (ii) holds in this case. 

 Now, assume that (ii) holds. Then Corollary \ref{corollary:FragShrewdWeaklyInacc} shows that $\kappa$ is regular. 
  Fix a closed unbounded subset $C$ of $\kappa$ and use Corollary \ref{corollary:KappaPlusEmbe} to find a transitive set $N$ and a non-trivial elementary embedding $\map{j}{\langle N,\in,R\rangle}{\langle\HH{\kappa^+},\in,R\rangle}$ such that   $j(\crit{j})=\kappa$ and $C\in\ran{j}$. 
 Then $C\cap\crit{j}\in N$, $j(C\cap\crit{j})=C$ and hence $\crit{j}\in\Lim(C)\subseteq C$. Since elementarity implies that $\crit{j}\in R$, we can now conclude that $C\cap R\neq\emptyset$. 
  This shows that (i) holds in this case. 
\end{proof}

\begin{corollary}\label{corollary:WeaklyInnMahloShrewd}
 \begin{enumerate}
  \item A cardinal $\kappa$ is weakly inaccessible if and only if it is weakly $(\Sigma_1,Cd,\kappa^+)$-shrewd. 
  
  \item A cardinal $\kappa$ is weakly Mahlo if and only if it is weakly $(\Sigma_1,Rg,\kappa^+)$-shrewd. 
 \end{enumerate}
\end{corollary}

\begin{proof}
 Note that a regular cardinal is weakly inaccessible if and only if it is a stationary limit of cardinals. Moreover, a regular cardinal is defined to be weakly Mahlo if and only it is a stationary limit of regular cardinals. Therefore both statement follow directly from Lemma \ref{lemma:ReflCardReg}. 
\end{proof}

Next, we show that weak $\Pi^1_n$-indescribability also corresponds to a certain canonical restrictions of weak shrewdness.

\begin{lemma}\label{lemma:WeakIndWeakShrewd}
 The following statements are equivalent for every cardinal $\kappa$ and every natural number $n>0$: 
 \begin{enumerate}
  \item $\kappa$ is weakly $\Pi^1_n$-indescribable. 
  
  \item $\kappa$ is weakly $(\Sigma_{n+1},\emptyset,\kappa^+)$-shrewd. 
 \end{enumerate}
\end{lemma}

\begin{proof}
 First, assume that (ii) holds, $A_0,\ldots,A_{m-1}$ are relations on $\kappa$ and $\Phi$ is a $\Pi^1_n$-sentence that holds in $\langle \kappa,\in,A_0,\ldots,A_{m-1}\rangle$. 
 Using Corollary \ref{corollary:KappaPlusEmbe}, we find a transitive set $N$ and a non-trivial elementary embedding $\map{j}{N}{\HH{\kappa^+}}$ with the property that $\crit{j}$ is a cardinal,  $N\prec_{\Sigma_n}\HH{\crit{j}^+}$, $j(\crit{j})=\kappa$ and $A_0,\ldots,A_{m-1}\in\ran{j}$. 
 Given $i<m$, if $A_i$ is a $k_i$-ary relation on $\kappa$, then have $A\cap\crit{j}^{k_i}\in N$ with $j(A\cap\crit{j}^{k_i})=A$. 
 Therefore, elementarity implies that, in $N$, the sentence $\Phi$ holds in the structure $$\langle\crit{j},\in,A_0\cap\crit{j}^{k_0},\ldots,A_{m-1}\cap\crit{j}^{k_{m-1}}\rangle.$$ 
  Since this statement 
   can be expressed by a $\Pi_n$-formula with parameters in $N$, our assumptions imply that it also holds in $\HH{\crit{j}^+}$ and therefore it holds in $\VV$ too. 
  These computations show that (i) holds in this case. 
  
  Now, assume that (i) holds. 
  In order to derive the desired conclusion, we introduce a canonical translation of $\Pi_n$-statements over $\HH{\kappa^+}$ into $\Pi^1_n$-statements over some expansions of the structure $\langle\kappa,\in\rangle$ by finitely-many relation symbols. 
  First, we set $$P ~ = ~ \Set{\langle\alpha,\beta,\goedel{\alpha}{\beta}\rangle}{\alpha,\beta\in\On}.$$ 
  For each infinite cardinal $\delta$, we now define $\CCCC(\delta)$ to consists of all subsets $B$ of $\delta$ with the property that $$E_B ~ = ~ \Set{\langle\alpha,\beta\rangle}{\alpha,\beta<\delta, ~ \goedel{\alpha}{\beta}\in B}$$ is an extensional and well-founded relation on $\delta$. 
  Note that with the help of rank functions, it is easy to see that there is a $\Sigma^1_1$-formula that uniformly defines $\CCCC(\delta)$ in $\langle\delta,\in,P\cap\delta^3\rangle$. 
  Given $B\in\CCCC(\delta)$, we let $\map{\pi_B}{\langle\delta,E_B\rangle}{\langle z_B,\in\rangle}$ denote the corresponding transitive collapse and set $x_B=\pi_B(0)$. Then $x_B,z_B\in\HH{\delta^+}$ for all $B\in\CCCC(\delta)$ and it is also easy to see that every element of $\HH{\delta^+}$ is of the form $x_B$ for some $B\in\CCCC(\delta)$. 
  Now, given an infinite cardinal $\delta$ and subsets $A,B_0,\ldots,B_{m-1}$ of $\delta$ with $B_0,\ldots,B_{m-1}\in\CCCC(\delta)$, we say that $B\in\CCCC(\delta)$ \emph{codes a model containing $A,x_{B_0},\ldots,x_{B_{m-1}}$} if the following statements hold: 
  \begin{itemize}
   \item $x_B=\kappa$, $\pi_B(1)=A$ and $\pi_B^{{-}1}(\alpha)=\omega\cdot(1+\alpha)$ for all $\alpha<\kappa$. 
   
   \item For every $i<m$, the transposition $\tau_i=(0,i+2)$\footnote{I.e. $\tau_i$ is the unique permutation of $\kappa$ with $\supp{\tau_i}=\{0,1+2\}$, $\tau_i(0)=i+2$ and $\tau_i(i+2)=0$.} is an isomorphism of $\langle\kappa,E_B\rangle$ and $\langle\kappa,E_{B_i}\rangle$. 
  \end{itemize}
  
  Note that, in the above situation, an easy induction shows that for all $i<m$, we have $\pi_B=\pi_{B_i}\circ\tau_i$,  $z_B=z_{B_i}$ and $x_{B_i}=\pi_B(i+2)$. 
  Moreover, it is easy to see that for every $m<\omega$, there is a $\Sigma_1^1$-formula $\Psi(v_0,\ldots,v_{m+1})$ with second-order variables $v_0,\ldots,v_{m+1}$ and the property that for every infinite cardinal $\delta$ and all $A,B,B_0,\ldots,B_{m-1}\subseteq\delta$, the statement $\Psi(A,B,B_0,\ldots,B_{m-1})$ holds in $\langle\delta,\in,P\cap\delta^3\rangle$ if and only if $B,B_0,\ldots,B_{m-1}\in\CCCC(\delta)$ and $B$ codes a model containing $A,x_{B_0},\ldots,x_{B_{m-1}}$. 
  Given some $\Sigma_1$-formula $\varphi(v_0,\ldots,v_m)$ in $\calL_\in$, we can now use the fact that a $\Sigma_1$-formula in $\calL_\in$  holds true if and only if it holds true in a transitive set containing all of its parameters  to find a $\Sigma_1^1$-formula $\Phi(v_0,\ldots,v_m)$ with second-order variables $v_1\ldots,v_m$ and the property that for  every infinite cardinal $\delta$, all $A\subseteq\delta$ and all $B_0,\ldots,B_{m-1}\in\CCCC(\delta)$, we have $$\HH{\delta^+}\models\varphi(A,x_{B_0},\ldots,x_{B_{m-1}}) ~ \Longleftrightarrow ~ \langle\delta,\in,P\cap\delta^3\rangle\models\Phi(A,B_0,\ldots,B_{m-1}).$$

  Now, fix a $\Pi_n$-formula $\varphi(v_0,v_1,v_2)$ and a subset $A$ of $\kappa$ with the property that the statement $\exists x ~ \varphi(x,A,\kappa)$ holds in $\HH{\kappa^+}$. 
   With the help of the above constructions, we can now find a $\Pi^1_n$-sentence $\Phi$ with the property that $$\HH{\delta^+}\models\exists x ~ \varphi(x,B,\delta) ~ \Longleftrightarrow ~ \exists X\subseteq\delta ~ \langle\delta,P\cap\delta^3,B,X\rangle\models\Phi$$ holds for every infinite cardinal $\delta$ and every subset $B$ of $\delta$. 
   Then there exists $X\subseteq\kappa$ with  $\langle\kappa,P\cap\kappa^3,A,X\rangle\models\Phi$ and hence we can apply {\cite[Theorem 6]{MR0281606}} to find a cardinal $\bar{\kappa}<\kappa$ such that $\Phi$ holds in $\langle\bar{\kappa},P\cap\bar{\kappa}^3,A\cap\bar{\kappa},X\rangle$. 
  But this shows that $\exists x ~ \varphi(x,A\cap\bar{\kappa},\bar{\kappa})$ holds in $\HH{\bar{\kappa}^+}$. 
  These computations allow us to conclude that (ii) holds in this case.  
\end{proof}

 We end this section by using the embedding characterization for weakly $\Pi^1_1$-indescribable cardinals provided by Corollary \ref{corollary:KappaPlusEmbe} and Lemma \ref{lemma:WeakIndWeakShrewd} to show that this large cardinal property is equivalent to Hamkins' \emph{weakly compact embedding property}. 
 Remember that a cardinal $\kappa$ has the weakly compact embedding property if for every transitive set $M$ of cardinality $\kappa$ with $\kappa\in M$, there is a transitive set $N$ and an elementary embedding $\map{j}{M}{N}$ with $\crit{j}=\kappa$.

 \begin{corollary}\label{corollary:WCEPSmallEmb}
  The following statements are equivalent for every infinite cardinal $\kappa$: 
  \begin{enumerate}
    \item $\kappa$ has the weakly compact embedding property.  
    
    \item $\kappa$ is weakly $\Pi^1_1$-indescribable. 
    
        \item There is a transitive set $N$ and a non-trivial elementary embedding $\map{j}{N}{\HH{\kappa^+}}$ with the property that $\crit{j}$ is a cardinal, $j(\crit{j})=\kappa$ and $N\prec_{\Sigma_1}\HH{\crit{j}^+}$.  
        
    \item For every cardinal $\theta>\kappa$ and all $z\in\HH{\theta}$, there is a transitive set $N$ and a non-trivial elementary embedding $\map{j}{N}{\HH{\theta}}$ with the property that $\crit{j}$ is a cardinal, $j(\crit{j})=\kappa$, $\HH{\crit{j}^+}^N\prec_{\Sigma_1}\HH{\crit{j}^+}$ and $z\in\ran{j}$.  
  \end{enumerate}
 \end{corollary}

\begin{proof}
 First, assume that (i) holds. Fix a cardinal $\theta>\kappa$ and  $z\in\HH{\theta}$. 
 Pick an elementary submodel $X$ of $\HH{\theta}$ of cardinality $\kappa$ with $\kappa\cup\{\kappa,z\}\subseteq X$ and let $\map{\pi}{\langle X,\in\rangle}{\langle B,\in\rangle}$ denote the corresponding transitive collapse. 
 Note that elementarity directly implies that $\pi^{{-}1}\restriction\HH{\kappa^+}^B=\id_{\HH{\kappa^+}^B}$. 
 Pick an elementary submodel $M$ of $\HH{\kappa^+}$ of cardinality $\kappa$ with $\kappa\cup\{B\}\subseteq M$ and fix a bijection $\map{b}{\kappa}{B}$ in $M$. 
  Since $M$ is transitive, we can use the weakly compact embedding property to find a transitive set $N$ and an elementary embedding $\map{j}{M}{N}$ with $\crit{j}=\kappa$. 
 Define $$D ~ = ~ \Set{\goedel{\alpha}{\beta}}{b(\alpha)\in b(\beta)} ~ \in ~ M$$ and let $E\in M$ be the binary relation on $\kappa$ coded by $D$. 
 Then $\langle \kappa,E\rangle$ is well-founded and extensional, and $\map{b}{\langle \kappa,E\rangle}{\langle B,\in\rangle}$ is the corresponding transitive collapse. 
 Now, since $D=j(D)\cap\kappa\in N$ implies that $E$ is contained in $N$ and $N$ is a transitive model of $\ZFC^-$, the above observations show that $B$ and $b$ are also elements of $N$. 
 Moreover, since $j\restriction\kappa=\id_\kappa$, we have $j(x)=(j(b)\circ b^{{-}1})(x)$ for all $x\in B$ and therefore we know that the elementary embedding $\map{j\restriction B}{B}{j(B)}$ is an element of $N$. 
  In addition, since we have $\HH{\kappa^+}^B\subseteq\HH{\kappa^+}^N\subseteq\HH{\kappa^+}$ and our setup ensures that $\HH{\kappa^+}^B=\HH{\kappa^+}^X$ is an elementary submodel of $\HH{\kappa^+}$, we know that  $\HH{\kappa^+}^B\prec_{\Sigma_1}\HH{\kappa^+}^N$. 
  Using elementarity, these computations yield a transitive set $A$ of cardinality less than $\kappa$ and an elementary embedding $\map{k}{A}{B}$ with the property that $\crit{k}$ is a cardinal with $k(\crit{k})=\kappa$, $\pi(z)\in\ran{k}$ and $\HH{\crit{k}^+}^A\prec_{\Sigma_1}\HH{\crit{k}^+}$. 
  Finally, using the fact that $\pi^{{-}1}(\kappa)=\kappa$, we can  conclude that the elementary embedding $\map{\pi^{{-}1}\circ k}{A}{\HH{\theta}}$ possesses all of the properties listed in (iv).  
  Therefore, these arguments show that (iv) holds in this case.

  Next, a combination of Lemma \ref{lemma:EmbLEvelShrewd} with Lemma \ref{lemma:WeakIndWeakShrewd} directly shows that (iv) implies (ii) and (ii) implies (iii). 
  
  Finally, assume, towards a contradiction, that (iii) holds and (i) fails. 
  By elementarity, we know that, in $N$, there is a transitive set $M$ of cardinality $\crit{j}$  with $\crit{j}\in M$ and the property that for every transitive set $B$, there is no elementary embedding $\map{k}{M}{B}$ with $\crit{k}=\crit{j}$. 
   %
 %
 Since this statement can be formalized by a $\Pi_1$-formula with parameters $M$ and $\crit{j}$, our assumptions imply that it holds in $\HH{\crit{j}^+}$ and, by $\Sigma_1$-absoluteness, it also holds in $\VV$. 
 But this yields a contradiction, because the map $\map{j\restriction M}{M}{j(M)}$ is an elementary embedding with these properties. 
\end{proof}

 Note that the above result also shows that one has to add inaccessibility to the assumptions of the statements in {\cite[Section 4]{SECFLC}} to obtain correct results.


\section{Characterizations of small large cardinals}

In this section, we will complete the proof of Theorem \ref{theorem:CharacterizationSmallLargeCardinals}. 
 We start by showing that all  local $\Sigma_n(R)$-classes considered in this section are in fact $\Sigma_2$-definable.

\begin{proposition}\label{proposition:SpecialCAseSigmanR}
 Let $R$ be a class, let $n>0$ be a natural number and let $S$ be a class that is  uniformly locally $\Sigma_n(R)$-definable in some parameter $z$. 
 \begin{enumerate}
  \item If $n=1$, then the class $S$ is $\Sigma_1(R)$-definable in the parameter $z$.  
  
  \item If  the class $PwSet$ is $\Sigma_1(R)$-definable in the parameter $z$, then the class $S$ is definable in the same way. 
  
  \item If  the class $R$ is definable by a $\Pi_1$-formula with parameter $z$, then the class $S$ is definable by a $\Sigma_2$-formula with parameter $z$. 
 \end{enumerate}
\end{proposition}

\begin{proof}
 Let $\varphi(v_0,v_1)$ be a $\Sigma_n$-formula in $\calL_{\dot{A}}$ witnessing that $S$ is uniformly locally $\Sigma_n(R)$-definable. 
 
 (i) If $n=1$, then the absoluteness of $\Sigma_0$-formulas between transitive structures implies that $S$ consists of all sets $x$ with the property that there exists a transitive set $N$ such that $x,z\in N$ and $\varphi(x,z)$ holds in $\langle N,\in,R\rangle$. 
  This equality then directly provides a $\Sigma_1(R)$-definition of $S$ in the parameter $z$. 
  
  (ii) If the class $PwSet$ is $\Sigma_1(R)$-definable in  parameter $z$, then the class $Cd$ of all cardinals and the function that sends a cardinal $\delta$ to the set $\HH{\delta^+}$ are definable in the same way. 
 It now follows that $S$ consists of all sets $x$ such that there exists a cardinal $\delta$ such that $x,z\in\HH{\delta^+}$ and $\varphi(x,z)$ holds in $\langle\HH{\delta^+},\in,R\rangle$. This again provides a $\Sigma_1(R)$-definition of $S$ in the parameter $z$. 
 
 (iii) Assume that there is a $\Pi_1$-formula $\psi(v_0,v_1)$ in $\calL_\in$ that witnesses that the class $R$ is $\Pi_1$-definable in the parameter $z$. 
  By $\Sigma_1$-absoluteness, the class $S$ consists of all sets $x$ with the property that there exists a cardinal $\delta$ and a set $Q$ such that $x,z\in\HH{\delta^+}$, $\varphi(x,z)$ holds in $\langle\HH{\delta^+},\in,Q\rangle$ and $$Q ~ = ~ \Set{y\in\HH{\delta^+}}{\HH{\delta^+}\models\psi(y,z)}.$$ 
  Since the function that sends a cardinal $\delta$ to the set $\HH{\delta^+}$ is definable by a $\Sigma_2$-formula without parameters, we can conclude that $S$ is  $\Sigma_2$-definable in the parameter $z$.  
\end{proof}

We now show how the restricted forms of weak shrewdness introduced above are connected to principles of structural reflection for local $\Sigma_n(R)$-classes.

\begin{lemma}\label{lemma:ReflUniformLocalFromEmb}
 Let $R$ be a class,  let $n>0$ be a natural number and let $\bar{\kappa}<\kappa$ be infinite cardinals. 
  If there exists a transitive set $N$  and an elementary embedding $\map{j}{\langle N,\in,R\rangle}{\langle\HH{\kappa^+},\in,R\rangle}$ with $\bar{\kappa}\in N$, $\langle N,\in,R\rangle\prec_{\Sigma_{n-1}}\langle\HH{\bar{\kappa}^+},\in,R\rangle$ and  $j(\bar{\kappa})=\kappa$, 
  then $j(\crit{j})$ is a regular cardinal and $\SRm{\calC}{j(\crit{j})}$ holds for every local $\Sigma_n(R)$-class over the set  $\Set{x\in N}{j(x)=x}$. 
\end{lemma}

\begin{proof}
 Set $\mu=j(\crit{j})$. 
  First, assume, towards a contradiction, that $\mu$ is singular. Then  $\crit{j}$ is singular in $N$, and, since $N$ is a transitive model of $\ZFC^-$, this shows that $N$ contains a cofinal function $\map{c}{\lambda}{\crit{j}}$ for some ordinal $\lambda<\crit{j}$. Then elementarity implies that the range of the function $\map{j(c)}{\lambda}{\mu}$ is cofinal in $\mu$ and this yields a contradiction, because elementarity also implies that $$\ran{j(c)} ~ = ~ \ran{c} ~ \subseteq ~ \crit{j} ~ < ~ \mu.$$

 Now, let $\calC$ be a local $\Sigma_n(R)$-class over the set $F=\Set{x\in N}{j(x)=x}$. 
 Pick a $\Sigma_n(R)$-formula $\varphi(v_0,v_1)$ and $z\in F$ witnessing that the class $\calC$ is uniformly locally $\Sigma_n(R)$-definable.  
   Assume, towards a contradiction, that $\calC$ contains a structure $A$ of cardinality $\mu$ with the property that for every structure $B$ in $\calC$ of cardinality less than $\mu$, there is no elementary embedding from $B$ into $A$.  
 Since $\calC$ is closed under isomorphic copies, this implies that, in $\langle\HH{\kappa^+},\in,R\rangle$, there is a structure $A$ of the given type of cardinality $\mu$ such that $\varphi(A,z)$ holds and for all structures $B$ of the given type of cardinality less than $\mu$, if $\varphi(B,z)$ holds, then there is no elementary embedding of $B$ into $A$. 
 By our assumptions, elementarity allows us to find a structure $A$ with these properties that is contained in $\ran{j}$. 
 Pick $B\in N$ with $j(B)=A$. Then elementarity ensures that $\varphi(B,z)$ holds in $\langle N,\in,R\rangle$. 
 Since $\Sigma_n$-formulas in $\calL_{\dot{A}}$ are upwards-absolute from $\langle N,\in,R\rangle$ to $\langle\HH{\bar{\kappa}^+},\in,R\rangle$, this shows that $\varphi(B,z)$ holds in $\langle\HH{\bar{\kappa}^+},\in,R\rangle$ and therefore $B$ is contained in $\calC$. 
  Next, since $N$ is transitive, the embedding $j$ induces an elementary embedding of $B$ into $A$. 
  Finally, since $B$ has cardinality less than $\mu$, we can pick an isomorphic copy $C$ of $B$ that is an element of $\HH{\kappa^+}$. 
  Then the structure $C$ is contained in $\calC$ and therefore $\varphi(C,z)$ holds in $\langle\HH{\kappa^+},\in,R\rangle$. But then the above arguments yield an elementary embedding from $C$ into $A$ and this map is also contained in $\HH{\kappa^+}$, contradicting the properties of $A$.   
\end{proof}

\begin{corollary}\label{corollary:ReflFromShrewd}
 Let $\kappa$ be an infinite cardinal, let $R$ be a class and let $n>0$ be a natural number. 
 If $\kappa$ is weakly $(\Sigma_n,R,\kappa^+)$-shrewd, then $\SRm{\calC}{\kappa}$ holds for every local $\Sigma_n(R)$-class $\calC$ over $\HH{\kappa}$. 
\end{corollary}

\begin{proof}
  Let $\calC$ be a local  $\Sigma_n(R)$-class  over $\HH{\kappa}$.  Pick $z\in\HH{\kappa}$ witnessing that the class $\calC$ is uniformly locally $\Sigma_n(R)$-definable. 
  Using Corollary \ref{corollary:KappaPlusEmbe}, our assumptions allow us to find a transitive set $N$ and a non-trivial elementary embedding $\map{j}{\langle N,\in,R\rangle}{\langle\HH{\kappa^+},\in,R\rangle}$ with the property that $\crit{j}$ is a cardinal, $\langle N,\in,R\rangle\prec_{\Sigma_{n-1}}\langle\HH{\crit{j}^+},\in,R\rangle$, $j(\crit{j})=\kappa$ and $z\in\ran{j}$. 
 Since we now have $z\in\HH{\crit{j}}\cap N$ and $j(z)=z$, Lemma \ref{lemma:ReflUniformLocalFromEmb} directly implies that $\SRm{\calC}{\kappa}$ holds.  
\end{proof}

The following partial converse of Lemma \ref{lemma:ReflUniformLocalFromEmb} is the last ingredient needed for our characterization of small large cardinals through principles of structural reflection:

\begin{lemma}\label{lemma:EmbFromRefl}
 Let $n>0$ be a natural number, let  $R$ be a class and let $z$ be a set such that  the class $Cd$ of all cardinals is uniformly locally $\Sigma_n(R)$-definable in the parameter $z$.  
  Then there is a local $\Sigma_n(R)$-class $\calC$ over $\{z\}$ with the property that whenever $\SRm{\calC}{\kappa}$ holds for some infinite cardinal $\kappa$ with $z\in\HH{\kappa}$, then for every $y\in\HH{\kappa^+}$, there exists a cardinal $\bar{\kappa}<\kappa$, a transitive set $N$ and an elementary embedding $\map{j}{\langle N,\in,R\rangle}{\langle\HH{\kappa^+},\in,R\rangle}$ such that $\bar{\kappa},z\in N$, $\langle N,\in,R\rangle\prec_{\Sigma_{n-1}}\langle\HH{\bar{\kappa}^+},\in,R\rangle$, $y\in\ran{j}$, $j(\bar{\kappa})=\kappa$ and $j(z)=z$.   
\end{lemma}

\begin{proof}
 Let $\calL$ denote the first-order language that extends $\calL_{\dot{A}}$ by three constant symbols. 
  Define $\calC$ to be the class of all $\calL$-structures that are isomorphic to an  $\calL$-structure  of the form $\langle N,\in,Q,\delta,y,z\rangle$ with the property that $\delta$ is an infinite cardinal, $N$ is a transitive set of cardinality $\delta$, $Q=N\cap R$ and $\langle N,\in,R\rangle\prec_{\Sigma_{n-1}}\langle\HH{\delta^+},\in,R\rangle$.

 \begin{claim*}
  $\calC$ is a local $\Sigma_n(R)$-class over $\{z\}$. 
 \end{claim*}
 
 \begin{proof}
  By definition, the class $\calC$ is closed under isomorphic copies. 
   Note that, given an infinite cardinal $\nu$ with $z\in\HH{\nu}$, an $\calL$-structure $A$ in $\HH{\nu^+}$ is contained in $\calC$ if and only if there exists a cardinal $\mu\leq\nu$, a transitive set $N$ of cardinality $\mu$ and  $y\in N$ such that $z\in N$, $\langle N,\in,R\rangle\prec_{\Sigma_{n-1}}\langle\HH{\mu^+},\in,R\rangle$ and $\HH{\nu^+}$ contains an isomorphism between $A$ and the resulting $\calL$-structure $\langle N,\in,N\cap R,\mu,y,z\rangle$. 
   In the case $n=1$, our assumptions on $R$ together with the fact that $\langle N,\in,R\rangle\prec_{\Sigma_0}\langle\HH{\betrag{N}^+},\in,R\rangle$ holds for every infinite transitive set $N$  directly yield a $\Sigma_1(R)$-formula witnessing that $\calC$ is uniformly locally $\Sigma_1(R)$-definable in the parameter $z$. 
  Now, assume that $n>1$. 
 Since the classes $\HH{\mu^+}$ are uniformly $\Sigma_1$-definable in the parameter $\mu$ and the relativization of a $\Sigma_n$-formula in $\calL_{\dot{A}}$ to a $\Sigma_1$-class again yields  a  $\Sigma_n$-formula in $\calL_{\dot{A}}$,
  we can use a universal $\Sigma_{n-1}$-formula in $\calL_{\dot{A}}$ to find a $\Sigma_n$-formula $\psi(v_0,v_1)$ in $\calL_{\dot{A}}$ with the property that for all infinite cardinals $\mu\leq\nu$ and all $N\in\HH{\nu^+}$, the statement $\psi(\mu,N)$ holds in $\langle\HH{\nu^+},\in,R\rangle$ if and only if $N\in\HH{\mu^+}$, $N$ is  transitive  and $\langle N,\in,R\rangle\prec_{\Sigma_{n-1}}\langle\HH{\mu^+},\in,R\rangle$. 
 Together with the above observations and our assumptions on $R$, we can again conclude that there  is a $\Sigma_n$-formula in $\calL_{\dot{A}}$ that witnesses that the class $\calC$ is uniformly locally $\Sigma_n(R)$-definable in the parameter $z$.   
 \end{proof}

 Now, let $\kappa$ be a cardinal with the property that $z\in\HH{\kappa}$ and $\SRm{\calC}{\kappa}$ holds. Fix $y\in\HH{\kappa}$ and pick an elementary submodel $\langle X,\in,R\rangle$ of $\langle\HH{\kappa^+},\in,R\rangle$ of cardinality $\kappa$ such that $\kappa\cup\{\kappa,y,z\}\subseteq X$. 
 Then the resulting $\calL$-structure $\langle X,\in,R\cap X,\kappa,y,z\rangle$ is  an element of  $\calC$ of cardinality $\kappa$. 
 By the definition of $\calC$ and our assumptions on $\kappa$, there exists an elementary embedding $j$ of a structure $\langle N,\in,Q,\bar{\kappa},x,z\rangle$ in $\calC$ of cardinality less than $\kappa$ with the property that $\bar{\kappa}$ is a cardinal,  $N$ is a transitive set, $Q=N\cap R$ and $\langle N,\in,R\rangle\prec_{\Sigma_{n-1}}\langle\HH{\bar{\kappa}^+},\in,R\rangle$. 
 In particular, we can conclude that $\map{j}{\langle N,\in,R\rangle}{\langle\HH{\kappa^+},\in,R\rangle}$ is an elementary embedding with $y\in\ran{j}$, $j(\bar{\kappa})=\kappa$ and $j(z)=z$.   
\end{proof}

\begin{corollary}\label{corollary:LeastReflectionPoint}
 Let $n>0$ be a natural number and let $R$ be a class with the property that  the class $Cd$   is uniformly locally $\Sigma_n(R)$-definable without parameters. 
 If $\kappa$ is the least cardinal with the property that $\SRm{\calC}{\kappa}$ holds for all local $\Sigma_n(R)$-classes over $\emptyset$, then $\kappa$ is weakly $(\Sigma_n,R,\kappa^+)$-shrewd. 
\end{corollary}

\begin{proof}
 Fix $z\in\HH{\kappa^+}$. 
  By Lemma \ref{lemma:EmbFromRefl}, there exists  a cardinal $\bar{\kappa}<\kappa$, a transitive set $N$ and an elementary embedding $\map{j}{\langle N,\in,R\rangle}{\langle\HH{\kappa^+},\in,R\rangle}$ with the property that $\bar{\kappa}\in N$,  $\langle N,\in,R\rangle\prec_{\Sigma_{n-1}}\langle\HH{\bar{\kappa}^+},\in,R\rangle$, $j(\bar{\kappa})=\kappa$ and $z\in\ran{j}$. 
  An application of Lemma \ref{lemma:ReflUniformLocalFromEmb} now shows that the minimality of $\kappa$ implies that $\crit{j}=\bar{\kappa}$. 
  By Lemma \ref{lemma:EmbLEvelShrewd}, these computations show that $\kappa$ is weakly $(\Sigma_n,R,\kappa^+)$-shrewd. 
\end{proof}

\begin{corollary}\label{corollary:CharaCLassesShrews}
 Let $n>0$ be a natural number and let $R$ be a class with the property that the class $Cd$ is uniformly locally $\Sigma_n(R)$-definable without parameters. 
 Then the following statements are equivalent for every cardinal $\kappa$: 
 \begin{enumerate}
  \item $\kappa$ is the least regular cardinal  with the property that $\kappa$ is weakly $(\Sigma_n,R,\kappa^+)$-shrewd. 
  
  \item $\kappa$ is the least cardinal with the property that $\SRm{\calC}{\kappa}$ holds for every local $\Sigma_n(R)$-class over $\emptyset$. 
  
  \item $\kappa$ is the least cardinal with the property that $\SRm{\calC}{\kappa}$ holds for every local $\Sigma_n(R)$-class over $\HH{\kappa}$. 
 \end{enumerate}
\end{corollary}

\begin{proof}
 First, assume that (i) holds. Then we can use Corollary \ref{corollary:ReflFromShrewd} to conclude that $\SRm{\calC}{\kappa}$ holds for every local $\Sigma_n(R)$-class $\calC$ over $\HH{\kappa}$. 
 Moreover, the minimality of $\kappa$ allows us to apply Corollary \ref{corollary:LeastReflectionPoint} to show that $\kappa$ is the least cardinal with the property that $\SRm{\calC}{\kappa}$ holds for all local $\Sigma_n(R)$-classes over $\emptyset$. 
 But this also shows that $\kappa$ is the least cardinal with the property that $\SRm{\calC}{\kappa}$ holds for all local $\Sigma_n(R)$-classes over $\HH{\kappa}$.

 Now, assume that (ii) holds. An application of Corollary \ref{corollary:LeastReflectionPoint} then shows that $\kappa$ is weakly $(\Sigma_n,R,\kappa^+)$-shrewd. Moreover, we can  use  Corollary \ref{corollary:ReflFromShrewd} to conclude that $\kappa$ is the least cardinal with this property.

  Finally, assume that (iii) holds and let $\nu\leq\kappa$ be the least cardinal with the property that $\SRm{\calC}{\nu}$ holds for every local $\Sigma_n(R)$-class over $\emptyset$. 
  Then the above computations show that $\nu$ is  weakly $(\Sigma_n,R,\nu^+)$-shrewd and Corollary \ref{corollary:ReflFromShrewd} implies that $\SRm{\calC}{\nu}$ holds for every local $\Sigma_n(R)$-class over $\HH{\nu}$. 
  This allows us to conclude that $\nu=\kappa$.  
\end{proof}

\begin{proof}[Proof of Theorem \ref{theorem:CharacterizationSmallLargeCardinals}]
 (i)  The desired equivalence is a direct consequence of a combination of Corollary \ref{corollary:WeaklyInnMahloShrewd} with Corollary \ref{corollary:CharaCLassesShrews}.

 (ii) Note that the class $Cd$ is uniformly locally $\Sigma_1(Rg)$-definable without parameters, because a limit ordinal is a cardinal if and only if it is either regular or a limit of regular cardinals. 
  Therefore, we can again combine Corollary \ref{corollary:WeaklyInnMahloShrewd} with Corollary \ref{corollary:CharaCLassesShrews} to derive the desired equivalence.

 (iii) Since the class $Cd$ is uniformly locally $\Sigma_2$-definable without parameters, the desired equivalence for weakly $\Pi^1_n$-indescribable cardinals is a direct consequence of Lemma \ref{lemma:WeakIndWeakShrewd} and Corollary \ref{corollary:CharaCLassesShrews}. 
\end{proof}

 We end this paper by studying restrictions on the reflection properties of various cardinals. 
  These results are motivated by unpublished work of Brent Cody, Sean Cox, Joel Hamkins and Thomas Johnstone that shows that various \emph{cardinal invariants of the continuum} do not possess the weakly compact embedding property (see \cite{HamkinsTalk}). 
  The following results place this implication into the general framework developed above.

 \begin{proposition}\label{proposition:MinIntervalDef}
  Given a class $I$ of infinite cardinals, there exists a class $\calC$ of structures of the same type such that $\SRm{\calC}{\min(I)}$ fails and the following statements hold for every natural number $n>0$, every class $R$ and every set $z$: 
  \begin{enumerate}
   \item If $I$ is $\Sigma_n(R)$-definable in the parameter $z$, then $\calC$ is definable in the same way. 
   
   \item If $I$ is uniformly locally $\Sigma_n(R)$-definable in the parameter $z$, then $\calC$ is a local $\Sigma_n(R)$-class over $\{z\}$. 
  \end{enumerate}
 \end{proposition}
 
 \begin{proof}
  Let $\calL$ denote the trivial first-order language and define $\calC$ to be the class of $\calL$-structures whose cardinality is an element of $I$. 
  Then $\SRm{\calC}{\min(I)}$  fails and both definability statements are immediate. 
 \end{proof}

  The next proposition provides examples of cardinal invariants of the continuum (the cardinality $2^{\aleph_0}$ of the continuum, \emph{bounding number $\mathfrak{b}$} and the \emph{dominating number $\mathfrak{d}$}) that can be represented as minima of universally locally $\Sigma_2$-definable classes of cardinals. 
 It should be noted that a combination of Lemma \ref{corollary:WeaklyInnMahloShrewd}, Corollary \ref{corollary:ReflFromShrewd} and Proposition \ref{proposition:MinIntervalDef} shows that, by starting with a Mahlo cardinal $\kappa$ and forcing \emph{Martin's Axiom} together with $2^{\aleph_0}=\kappa$ to hold in a generic extension, it is possible to obtain a model in which the set $\{2^{\aleph_0}\}=[\mathfrak{b},2^{\aleph_0}]=[\mathfrak{d},2^{\aleph_0}]$ is not uniformly locally $\Sigma_1(Rg)$-definable with parameters in $\HH{\kappa}$.

 \begin{proposition}
  The sets $\{2^{\aleph_0}\}$, $[\mathfrak{b},2^{\aleph_0}]$ and $[\mathfrak{d},2^{\aleph_0}]$ are all uniformly locally $\Sigma_2$-definable without parameters. 
 \end{proposition}
 
 \begin{proof}
  First, let $\varphi(v)$ be the canonical $\Sigma_2$-formula stating that $v$ is a cardinal and there exists a bijection between $v$ and $\POT{\omega}$. 
  Fix an infinite cardinal $\nu$. In one direction, if $2^{\aleph_0}\leq\nu$, then $\HH{\nu^+}$ contains a bijection between $2^{\aleph_0}$ and the reals, and hence $\varphi(2^{\aleph_0})$ holds in $\HH{\nu^+}$. 
  In the other direction, if $\varphi(\mu)$ holds in $\HH{\nu^+}$, then the fact that $\HH{\nu^+}$ contains all  reals implies that $2^{\aleph_0}=\mu$. 
  This shows that $\varphi(v)$ witnesses that the set $\{2^{\aleph_0}\}$ is uniformly locally $\Sigma_2$-definable without parameters. 
  
  Next, let $\varphi(v)$ be the canonical $\Sigma_2$-formula stating that $v$ is a cardinal and there exists an unbounded family of cardinality $v$ in $\langle{}^\omega\omega,<_*\rangle$. 
  Fix an infinite cardinal $\nu$. Given a cardinal $\mathfrak{b}\leq\mu\leq\min(\nu,2^{\aleph_0})$, the set $\HH{\nu^+}$ contains a bijection between $\mu$ and an unbounded family in $\langle{}^\omega\omega,<_*\rangle$, and hence we know that $\varphi(\mu)$ holds in $\HH{\nu^+}$. 
  In the other direction, if $\varphi(\mu)$ holds in $\HH{\nu^+}$, then there exists an unbounded family in $\langle{}^\omega\omega,<_*\rangle$ of cardinality $\mu$ and hence $\mathfrak{b}\leq\mu\leq 2^{\aleph_0}$, because ${}^\omega\omega$ is a subset of $\HH{\nu^+}$. Therefore the formula $\varphi(v)$ witnesses the desired definability of the set $[\mathfrak{b},2^{\aleph_0}]$.
  
  Finally, if $\varphi(v)$ denotes the canonical $\Sigma_2$-formula stating that $v$ is a cardinal and there exists a dominating family of cardinality $v$ in $\langle{}^\omega\omega,<_*\rangle$, then we may argue as above to show that  $\varphi(v)$ witnesses the desired definability of the set $[\mathfrak{d},2^{\aleph_0}]$. 
 \end{proof}

 In combination with Lemma \ref{lemma:WeakIndWeakShrewd}, Corollary \ref{corollary:WCEPSmallEmb} and Corollary \ref{corollary:ReflFromShrewd}, the above propositions provide a general reason for the fact that the cardinals $2^{\aleph_0}$, $\mathfrak{b}$ and $\mathfrak{d}$ do not have the weakly compact embedding property. 
 %
 In contrast, as discussed in the introduction, Hamkins observed that a cardinal with the weakly compact embedding property can be the predecessor of  $2^{\aleph_0}$. 
   %
  Now, note that if $I$ is a class of cardinals that is definable by a $\Sigma_2$-formula with parameter $z$ and $\beta$ is an ordinal, then the class of all cardinals of the form $\aleph_\alpha$ with the property that the cardinal $\aleph_{\alpha+\beta}$ is an element of $I$ is  definable by a  $\Sigma_2$-formula with parameters $\beta$ and $z$. 
  In particular, the above propositions show that if $\kappa$ is a weakly shrewd cardinal and $\alpha<\kappa$, then $\kappa^{{+}\alpha}\notin\{2^{\aleph_0},\mathfrak{b},\mathfrak{d}\}$. 
 Finally, note that a combination of Theorem \ref{theorem:HigherReflection}, Lemma \ref{lemma:SubtleHyperShrewd} and Lemma \ref{lemma:AddGenExtHyperShrew} shows that, if $\delta$ is a subtle cardinal and $G$ is $\Add{\omega}{\delta}$-generic over $\VV$,
 then, in $\VV[G]$, the interval $(\mathfrak{b},\mathfrak{d})$\footnote{Note that $\mathfrak{b}=\aleph_1$ and $2^{\aleph_0}=\mathfrak{d}=\delta$ holds in $\VV[G]$ (see {\cite[Section 11.3]{MR2768685}}).}  contains unboundedly many weakly shrewd cardinals. 


 \bibliographystyle{amsplain}
\bibliography{references}


\end{document}